\documentclass[
12pt,  reqno]{amsart}
\pdfoutput=1
\usepackage[margin=1.2in,marginparwidth=1.5cm, marginparsep=0.5cm]{geometry}
\usepackage{mathtools}

\usepackage{booktabs} 
\usepackage{microtype}
\usepackage{amssymb}
\usepackage{mathrsfs}

\usepackage{xcolor}
\usepackage[implicit=true]{hyperref}

\usepackage{soul}
\usepackage{yhmath}

\usepackage{xsavebox}

\usepackage{amsrefs}

\usepackage{cases}

\allowdisplaybreaks[2]

\newtheorem{theorem}{Theorem} [section]

\newtheorem{lemma}[theorem]{Lemma}
\newtheorem{proposition}[theorem]{Proposition}
\newtheorem{remark}[theorem]{Remark}

\newtheorem{definition}[theorem]{Definition}

\DeclareMathOperator{\Law}{Law}

\newcommand{\I}{\mathcal{I}}

\newcommand{\Ha}{\mathbb{H}_a}

\newcommand{\noi}{\noindent}
\newcommand{\Z}{\mathbb{Z}}
\newcommand{\R}{\mathbb{R}}

\newcommand{\T}{\mathbb{T}}

\newcommand{\dr}{\theta}
\newcommand{\Dr}{\Theta}

\newcommand{\W}{\mathcal{W}}

\let\P= \undefined
\newcommand{\P}{\mathbf{P}}

\newcommand{\E}{\mathbb{E}}

\newcommand{\F}{\mathcal{F}}

\newcommand{\al}{\alpha}
\newcommand{\be}{\beta}
\newcommand{\dl}{\delta}

\newcommand{\nb}{\nabla}

\newcommand{\Dl}{\Delta}
\newcommand{\eps}{\varepsilon}

\newcommand{\g}{\gamma}

\newcommand{\s}{\sigma}

\newcommand{\ft}{\widehat}

\newcommand{\cj}{\overline}

\newcommand{\dt}{\partial_t}
\newcommand{\dd}{\partial}

\renewcommand{\l}{\ell}
\renewcommand{\o}{\omega}
\renewcommand{\O}{\Omega}

\newcommand{\les}{\lesssim}

\newcommand{\jb}[1]
{\langle #1 \rangle}

\newcommand{\N}{\mathbb{N}}

\newcommand{\Pkn}{P^{(2m+2)}_{\al, N}}
\renewcommand{\H}{\mathcal{H}}

\newtheorem*{ackno}{Acknowledgements}

\numberwithin{equation}{section}
\numberwithin{theorem}{section}

\begin{document}

\title[Invariant Gibbs dynamics for fractional NLW]{Invariant Gibbs dynamics for two-dimensional fractional wave equations in negative Sobolev spaces}

\author[L. Forcella and O. Pocovnicu
]{Luigi Forcella and Oana Pocovnicu}

\address{Luigi Forcella\\Department of Mathematics, University of Pisa, Largo Bruno Pontecorvo 5, 56127, Pisa, Italy}
\email{luigi.forcella@unipi.it}

\address{Oana Pocovnicu, Department of Mathematics, Heriot-Watt University and The Maxwell Institute for the Mathematical Sciences, Edinburgh, EH14 4AS, United Kingdom}

\email{o.pocovnicu@hw.ac.uk}



\begin{abstract}
We consider a fractional nonlinear wave equation (fNLW) with a general odd power-type nonlinearity, on the two-dimensional torus. Our main goal is to construct invariant global-in-time Gibbs dynamics for a renormalized version of fNLW. We first construct the renormalized Gibbs measure associated with the renormalized equation by using the variational approach of Barashkov and Gubinelli. We then prove almost sure local well-posedness with respect to Gibbsian initial data, by exploiting the second order expansion. Finally, we extend solutions globally in time using Bourgain's invariant measure and approximation argument. 
\end{abstract}
	
\maketitle

\section{Introduction}	\setcounter{equation}{0}
We consider 
the defocusing fractional nonlinear wave equations
 (fNLW) posed on the two-dimensional flat torus:
\begin{align}
\begin{cases}
\dt^2 u +(1 -\Dl)^\alpha  u +  u^{2m+1} = 0 \\
(u, \dt u) |_{t = 0} = (\phi_0,  \phi_1), 
\end{cases}
\qquad (t,x) \in \R \times \T^2,
\label{NLW}
\end{align}

\noi
where $\al\in(0,1),$ 
$m \in \N $ is a positive integer, and $\T^2 = (\R/\Z)^2$.
Despite the presence of the mass term, we refer to \eqref{NLW}
as fNLW,  though it is usually referred  to as the nonlinear Klein-Gordon equation.
We restrict our attention to the case of real-valued solutions $u:\R\times\T^2\mapsto \R$.\medskip

The purpose of this paper is twofold. First, we are interested in the construction of a (invariant) Gibbs measure for a renormalized version of equation \eqref{NLW}. Secondly, we prove almost sure global well-posedness of the renormalized fNLW with initial data distributed according to the Gibbs measure, in Sobolev spaces of negative regularity. 
The a.s. global well-posedness is a consequence of an a.s. local well-posedness theory, followed by the application of Bourgain's invariant  measure and approximation argument. 
Hence, in the second part of the paper we will focus on the almost sure local well-posedness theory. 

Both the construction of renormalized Gibbs measures and the local well-posedness of the renormalized fNLW equation, strongly depend on the value $\al$ of the fractional power of the  elliptic operator $1-\Delta$ appearing in \eqref{NLW}.  
\subsection{Renormalized Gibbs measures}
With the notation $v = \dt u$, 
we can rewrite the equation \eqref{NLW} 
as the following Hamiltonian evolution:
\begin{equation}\label{ham-form}
 \dt 
 \begin{pmatrix}
 u \\ v 
 \end{pmatrix}
 =  
 \begin{pmatrix}
 0& 1 \\ -1 & 0
 \end{pmatrix}
\frac{\dd H_{\al}}{\dd(u, v )}, 
\end{equation}

\noi
associated to the energy $H_{\al} = H_{\al}(u, v)$ defined by 
\begin{align}
H_{\al}(u, v): 
 = \frac{1}{2}\left(\|u\|_{H^{\al}_x(\T^2)}^2+\|v\|_{L^2_x(\T^2)}^2\right)
+ \frac1{2m + 2} \|u\|^{2m+2}_{L^{2m+2}_x(\T^2)}.
\label{Hamil}
\end{align}

\noi In the case of a Hamiltonian system on $\R^{2n}$,
the corresponding Gibbs measure is invariant. 
This follows from (i) the conservation of the energy/Hamiltonian and (ii) from the invariance of the finite dimensional Lebesgue measure under the flow of a Hamiltonian system on $\R^{2n}$ (which is a consequence of Liouville's theorem). Mimicking the picture of the finite dimensional setting, one is tempted to define in a similar way the Gibbs measure for \eqref{NLW}:
\begin{align}
dP^{2m+2}_{\alpha,2}``= "Z^{-1} \exp(- \be H_{\al}(u,v ))\,du\otimes dv.
\label{G1}
 \end{align}
 
\noi 
However, \eqref{G1} is merely a formal expression and we need to rigorously construct a Gibbs measure for (a renormalized version of) fNLW; this is one of the main tasks of this article. As in the finite dimensional setting, we also would like to have invariance of the (renormalized) Gibbs measure under the dynamics of (the renormalized version) fNLW. In \eqref{G1}, and in similar expression along the article, $Z$,  $Z_0$, $Z_1$ etc. denote various normalizing constants
so that the corresponding measures are probability measures. The parameter $\be$ stands for the reciprocal temperature, and we fix $\be=1$.  All the results in this paper hold, in fact, for any $\be > 0$ and the resulting (renormalized) Gibbs measures are mutually singular for different values of $\be > 0$, see~\cite{OQ}. 

From \eqref{Hamil}, we can rewrite the formal expression~\eqref{G1}
as 
\begin{align}
dP^{2m+2}_{\al,2}
& ``= "Z^{-1} 
e^{-\frac 1{2m+2} \int u^{2m+2} \,dx} 
e^{-\frac 12 \int |(1-\Dl)^{\al/2} u|^2 \,dx} du\otimes e^{-\frac 12 \int v^2\,dx} dv \notag\\
& \sim e^{-\frac 1{2m+2} \int u^{2m+2}\,dx} d\vec \mu_{\al} , 
\label{G2}
 \end{align}

\noi
where $\vec\mu_{\al}$ is  the Gaussian measure  on $\mathcal{D}'(\T^2)\times \mathcal{D}'(\T^2)$  with  density 
\begin{align}
 d\vec \mu_{\al}
  =  Z^{-1} e^{-\frac 12 \int|(1-\Dl)^{\al/2} u|^2\, dx}\,du\otimes e^{-\frac 12 \int v^2\,dx}\,dv.
\label{G3}
\end{align}

\noi
Note that $\vec \mu_{\al}$ has a tensorial structure $\vec \mu_{\al} = \mu_{\al} \otimes \nu$, where
$\mu_{\al}$ and~$\nu$ are given by 
\begin{align}
 d\mu_{\al} =  Z_0^{-1} e^{-\frac 12 \int |(1-\Dl)^{\al/2} u|^2\,dx}\,du
 \qquad \text{and}\qquad 
 d\nu =  Z_1^{-1} e^{-\frac 12 \int v^2\,dx}\,dv
\label{G4}
\end{align}

\noi
Namely, $\mu_\al$ is the Ornstein-Uhlenbeck measure
and $\nu$ is the white noise measure on $\T^2$.
When $\al=1$, $\mu_1$ corresponds to the massive Gaussian free field. 
It is well known that $\mu_1$ is supported on 
the $L^2$-based Sobolev space $H^s(\T^2)$, $s<0$.
Note also that 
$\vec\mu_{\al}$ is the induced probability measure
under the map
\begin{align}
 \o \in \O \longmapsto (u^\o, v^\o)(x) 
 = \bigg( \sum_{n \in \Z^2} \frac{g_{0, n}(\o)}{\jb{n}^{\al}}e^{in\cdot x},
\sum_{n \in \Z^2} g_{1, n}(\o)e^{in\cdot x}\bigg), 
\label{G5}
 \end{align}

\noi
where 
$\jb{n} = \sqrt{1+|n|^2}$,  and 
$\{g_{0, n}, g_{1, n}\}_{n \in \Z^2}$
is a sequence of independent standard complex-valued Gaussian
random variables on a probability space $(\O, \F, P)$
satisfying the condition  ${g_{j, -n} = \cj{g_{j, n}}}$ for  $j = 0, 1$.
More precisely, with the notation $\N=\{1,2,3,\dots\}$, we first define an index set $\Lambda$ by
\[\Lambda:=(\Z\times \N) \cup (\N\times\{0\})\cup \{(0,0)\}.\]
We then define $\{g_{0,n},g_{1,n}\}_{n\in\Lambda}$ 
to be a family of independent standard Gaussian random variables which are complex-valued for $n\neq (0,0)$ and are real-valued for $n=(0,0)$. 
We finally set $g_{j,n}=\cj{g_{j,-n}}$ for $n\in\Z^2\setminus \Lambda$ and $j = 0, 1$.

From \eqref{G5}, we see that $\vec \mu_{\al}$ is supported
on \[\H^s_\al(\T^2) = \H^s (\T^2) : = H^s(\T^2)\times H^{s-\al}(\T^2) , \quad s < \al-1.\]  
The claim on the support of the measure essentially follows by calculating 
\[\E (\|(u^\o, v^\o)\|_{\H^s}^2)\sim \sum_{n\in\Z^2}\jb{n}^{2(s-\al)}\]
and by noting that the series on the right hand-side is convergent provided that $s<\al-1$.
Moreover, we have that $\vec\mu_\al (\H^{\al-1})=0$. 
Since $\al<1$, this implies that $\int_{\T^2} (u^{\o})^2\, dx = \infty$, and so
$\int_{\T^2} (u^\o)^{2m+2}\, dx = \infty$,
$\vec\mu_{\al}$-almost surely.
Consequently, the right-hand side of \eqref{G2} cannot define a probability measure.
We are therefore forced to {\it renormalize} the potential part of the energy.
In the 2D setting, it is known that a Wick ordering suffices for this purpose.
We refer the reader to \cites{Simon,GJ, OT2018, OT2020}. \smallskip

Let us briefly go over  the Wick renormalization on $\T^2$.
See~\cite{DPT} for more details on  $\T^2$, which is the  framework of the present article.
For a typical element $u$ under  $\mu_{\al}$ defined in~\eqref{G4},  since $u \notin L^2(\T^2)$ almost surely, we have that 
\[ \int_{\T^2} u^2\,dx = \lim_{N \to \infty} \int_{\T^2} (\P_Nu)^2\,dx = \infty\]
almost surely,
where 
 $\P_Nu$ is the Fourier projector onto  frequencies at most $N$, namely 
 $\P_Nf:=\sum_{n \in \Z^2, |n|\leq N}\ft f(n)e^{in\cdot x}$. \smallskip
Moreover, we observe that for each $x \in \T^2$,  
 $\P_N u(x)$ 
 is a mean-zero real-valued Gaussian random variable with variance 
 \begin{align}
\s_{\al,N}:= \E [ (\P_N u)^2(x)]
= \sum_{\substack{n \in\Z^2\\ |n|\leq N}}\frac{1}{{\jb n}^{2\al}}.
\label{G6}
\end{align}
Thus $\s_{\al,N}$ diverges as $N^{2-2\al}$ as $N\to\infty$. Observe that the quantity $\s_{\al, N}$ introduced in  \eqref{G6} is independent of $x \in \T^2$.
One then defines the Wick ordered monomial
 $:\! (\P_N u)^k  \!:$ by
\begin{align}
:\! (\P_N u)^k (x)\!:  \, := H_k(\P_N u(x); \s_{\al,N}),
\label{Wick1}
\end{align}

\noi where  $H_k(x; \s)$ is the Hermite polynomial of degree $k$
defined via the generating function
\begin{equation*}
F(t, x; \s) : =  e^{tx - \frac{1}{2}\s t^2} = \sum_{k = 0}^\infty \frac{t^k}{k!} H_k(x;\s).
 \end{equation*}
	
\noi
Here we set 
$H_k(x) : = H_k(x; 1)$.
A list of the first  Hermite polynomials is given for the sake of clarity
and for readers' convenience:
\begin{align}
\begin{split}
& H_0(x; \s) = 1, 
\qquad 
H_1(x; \s) = x, 
\qquad
H_2(x; \s) = x^2 - \s, \\
& H_3(x; \s) = x^3 - 3\s x, 
\qquad 
H_4(x; \s) = x^4 - 6\s x^2 +3\s^2.
\end{split}
\label{H1a}
\end{align}
Furthermore, we report some useful identities:
\begin{equation}\label{eq:hermite:hom}
H_k(\s^{1/2}x;\s)=\s^{k/2}H_k(x),
\end{equation}
\begin{align}
H_k(x+y; \s )
&  = 
\sum_{\l = 0}^k
\begin{pmatrix}
k \\ \l
\end{pmatrix}
 x^{k - \l} H_l(y; \s),
\label{Herm:sum}
\end{align}
and
\begin{equation}\label{partial-herm}
\partial_x H_k(x; \s) = k H_{k-1}(x; \s).
\end{equation}


Having defined 
$:\! (\P_N u)^{2m+2} (x)\!: \, = H_{2m+2}(\P_N u(x); \s_{\al,N})$, we introduce
$G_N(u)$ and $R_N(u)$ defined by
\begin{equation}\label{eq:def:GN}
G_N(u):= \int_{\T^2} :\! (\P_N u)^{2m+2}\! : dx
\end{equation}
and 
\begin{equation}\label{eq:def:RN}
R_N(u):=e^{-\frac 1{2m+2} \int_{\T^2} :  (\P_N u)^{2m+2}  : \, dx} =e^{-\frac {1}{2m+2}G_N(u)}.
\end{equation}

\noi
Then, the following proposition allows us to define $\int_{\T^2} :\! u^{2m+2} \!:dx$.
\begin{proposition}\label{prop:conv:G}
Let $m \in \N$ and let $\al\in\left(1- \frac{1}{2m+2},1\right)$. Let $1\leq p\leq \infty$. 
Let $u=u^\omega$ be a typical element under the measure $\mu_{\al}$ defined in~\eqref{G4}.
The sequence $\{G_N(u)\}_{N\in\N}$ is a Cauchy sequence in $L^p(\mu_{\al})$. More specifically, there exists a constant $c=c(m)$ such that for any $1\leq N\leq M$ we have
\begin{equation}\label{eq:cauchy:G}
\|G_M(u)-G_N(u)\|_{L^p(\mu_\al)} \leq c(m)(p-1)^{m+1} N^{ 1-2\al+\frac{m}{m+1}}.\end{equation}
We therefore define 
\begin{equation*}
G(u)=\int_{\T^2} :\,\! u^{2m+2} \!\,:dx:\,=L^p\hbox{-}\lim_{N\to\infty}\int_{\T^2} :\! (\P_N u)^{2m+2} \! : \, dx.
\end{equation*}
\end{proposition}

Note that if $\al>1-\frac{1}{2m+2}$, the exponent of $N$ in \eqref{eq:cauchy:G} is negative, hence $\{G_N(u)\}_{N\in\N}$ is indeed a Cauchy sequence. \\

The next theorem is the core result for the construction of the renormalized Gibbs measure. 

\begin{theorem}\label{thm:exp:int}
Let $m \in \N$ and let $\al\in\left(1- \frac{1}{2m+2},1\right).$ Let $1\leq p\leq \infty$. 
Let $u=u^\omega$ be a typical element under the measure $\mu_{\al}$ defined in~\eqref{G4}.
Then 
$\{R_N(u)\}_{N\in\N}$ is uniformly bounded in $L^p(\mu_\al)$. 
Moreover, $R_N(u)$ convergence in $L^p(\mu_\al)$ to
\[R(u):\, = e^{-\frac{1}{2m+2}\int_{\T^2} :\,\! u^{2m+2} \!\,:\,dx:}.\]
\end{theorem}

With the above theorem and with the intuition gained from \eqref{G2}, we can rigorously define the renormalized Gibbs measure  associated to the `Wick ordered Hamiltonian'
\begin{align}
H^{W}_{\al}(u, v) &: = \frac{1}{2}\int_{\T^2}\,|\left(1-\Dl \right)^{\al/2} u|^2\, dx
+ 
\frac{1}{2}\int_{\T^2} v^2dx
+ \frac1{2m + 2} \int_{\T^2} :\! u^{2m+2} \!:  dx
\label{ham-wick}
\end{align}
by
\[
dP^{2m+2}_{\al,2} : \,
=Z^{-1}e^{-H^W_\al(u,v)}du \otimes dv
= Z^{-1} 
 e^{-\frac 1{2m+2} \int : \, u^{2m+2}\,: \,dx}\, d\vec \mu_{\al}.
\]
\noi
$P^{2m+2}_{\al,2}$ is a probability measure on $\mathcal H^s(\T^2)\setminus \mathcal H^{\al-1}(\T^2)$, $s<\al-1$. (Note that, so far, $H^{W}_{\al}(u, v)$ was only defined for $u$ a typical element in the support of $\mu_\al$. )

The proof of Proposition \ref{prop:conv:G} and that of Theorem \ref{thm:exp:int} are the main contents of Section \ref{sec:variational}.

\begin{remark}\rm
In the case $m=1$ of the cubic fNLW, the result in Theorem \ref{thm:exp:int} as well is its proof overlap with a result by by Sun, Tzvetkov, and Xu from \cite{STX}. We proved the result above independently, at the same time. Additionally, it is worth mentioning a few differences between \cite{STX} and this paper.\smallskip

\noi
\textup{(i)} Our goal in this paper is different from the goal in \cite{STX}. Namely, the aim of \cite{STX}  is to prove the weak universality of the 2D (renormalized) cubic  fNLW (both the convergence of invariant measures and the dynamical weak universality). In this paper, we consider a 2D (renormalized)  fNLW with a general odd power-type nonlinearity, we construct the renormalized Gibbs measure, and we prove its almost sure global well-posedness with respect to  Gibbsian initial data (see Theorem \ref{THM:GWP} and Theorem \ref{THM:GWP:sec-ord} below).

\smallskip
\noi
\textup{(ii)}  In \cite{STX}, the authors consider fNLW with a nonlinearity given by a polynomial with prescribed properties; more precisely, the nonlinearity is a linear combination of cubic and higher powers which, roughly speaking, scale as the cubic power. In particular, from the point of view of the measure construction, the nonlinearity in \cite{STX} is essentially of cubic type. Consequently, the authors find a lower bound $\al>\frac34$ for the construction of the renormalized Gibbs measure which is independent of the degree of the nonlinearity. In our paper, we consider general defocusing power-type nonlinearities $u^{2m+1}$, and thus the restriction we have on $\al$ for the construction of the Gibbs measure does depend on the power of the nonlinearity, namely on $m$. See the condition $\al>1-\frac{1}{2m+2}$ in Proposition \ref{prop:conv:G} and Theorem \ref{thm:exp:int}. This restriction on $\alpha$ coincides with that in \cite{STX} when $m=1$ (cubic nonlinearity).  
\end{remark}

\subsection{Solution theory for renormalized fNLW} 
We now introduce a notion of solution for the Cauchy problem  with random initial data. 

\subsubsection{First order expansion}

We first recall the integral formulation of the fractional NLW equation, i.e. we formally write a solution to \eqref{NLW} as
\[
u(t) = 
\cos (t \jb{\nb}^\al) \phi_0
+ \frac{\sin (t \jb{\nb}^\al)}{\jb{\nb}^\al}   \phi_1+\int_0^t \frac{\sin ((t - t') \jb{\nb}^\al)}{\jb{\nb}^\al} 
u^{2m+1} dt'.
\]
The formulation above is the well-known Duhamel formulation of \eqref{NLW}, and we further denote by $\I$ the Duhamel operator 
\begin{equation}\label{duha-oper}
\I(F) :=\int_0^t \frac{\sin ((t - t') \jb{\nb}^\al)}{\jb{\nb}^\al} 
F(t') dt'.
\end{equation}
By also denoting by $z$ the solution to the linear fractional wave equation with intial condition $(\phi_0,\phi_1)$:
\begin{equation}\label{lin1}
z(t)=S(t)(\phi_0,\phi_1):=\cos (t \jb{\nb}^\al) \phi_0
+ \frac{\sin (t \jb{\nb}^\al)}{\jb{\nb}^\al}   \phi_1,
\end{equation}
we see that Duhamel's formula states that 
solutions to \eqref{NLW} are written $u=z+\I(u^{2m+1})$.

 From Proposition \ref{prop:conv:G} and Theorem \ref{thm:exp:int}, we can consider the Wick ordered energy defined in \eqref{ham-wick}
and, analogously to \eqref{ham-form}, we can consider the Hamiltonian formulation of the equation with the Wick ordered Hamiltonian \eqref{ham-wick}, with initial data $(\phi_0^\o,  \phi_1^\o)$ distributed according to $P^{2m+2}_{\al,2}$, namely of the form \eqref{G5}. 
Note that the potential part of the Wick ordered Hamiltonian, $ \int_{\T^2} :\! u^{2m+2} \!:  dx $
has only been defined so far for $u$ distributed according to $P^{2m+2}_{\al,2}$. 
This definition can be extended to a wider class of functions of the form $z^\o+w$, where
$z^\o(t)=S(t)(\phi^\o_0,\phi_1^\o)$
and $w$ is a ``nice" function of positive regularity.
Indeed, for any $k\in\N$, from the property of the Hermite polynomials \eqref{Herm:sum}, 
we have
\begin{equation}\label{herm-bin}
 :\! (z_N+w_N)^{k}\!\!:  \, \, 
  = \sum_{\l = 0}^{k}
\begin{pmatrix}
k \\ \l
\end{pmatrix}
  :\! z_N^\l\!: w_N^{k - \l},
\end{equation}
where $z_N:=\P_Nz^\o=\P_N S(t)(\phi_0^\o,\phi_1^\o)$, $w_N:=\P_Nw$ and $:\! z_N^\l\!:  = H_\l(z_N; \sigma_{\al,N})$.
Then, using 
part (i) of Proposition \ref{PROP:sto1} below, 
we have that
$\{:\! z_N^\l\!: \}_N$ is a Cauchy sequence in $C_t(\R_+;W_x^{-\l(1-\al)-\eps, \infty}(\T^2))$
almost surely, for any $\eps>0$, and so converges to
\[  :\! z^\l\!:\,:=\lim_{N\to\infty} :\! z_N^\l\!: \textrm{ in } C_t(\R_+;W_x^{-\l(1-\al)-\eps, \infty}(\T^2)).\]
For $u=z^\o+w$, we can then define the Wick ordered monomial $:\!u^k\!:\,$ by:
\begin{equation*}
:\!u^k\!:\,:= :\! (z^\o+w)^{k}\!\!:  \, \, 
  = \sum_{\l = 0}^{k}
\begin{pmatrix}
k \\ \l
\end{pmatrix}
  :\! z^\l\!: w^{k - \l}.
\end{equation*}

\noi
Therefore, we can define $ \int_{\T^2} :\! u^{2m+2} \!:  dx $ in the Wick ordered energy
for $u$ of the form $u=z^\o+w$ for $w$ of positive regularity. 
With this definition and recalling Duhamel's formula,  
we arrive at the Hamiltonian evolution corresponding to the 
Wick ordered energy:
\begin{align}
\begin{cases}
\dt^2 u+(1 -\Dl)^\alpha  u \, +  :\! u^{2m+1} \!:\, = 0 \\
(u, \dt u) |_{t = 0} = (\phi_0^\o, \phi_1^\o), 
\end{cases}
\label{WNLW2}
\end{align}

\noi
where 
\[ 
(\phi_0^\o, \phi_1^\o)= \bigg( \sum_{n \in \Z^2} \frac{g_{0, n}(\o)}{\jb{n}^{\al}}e^{in\cdot x},
\sum_{n \in \Z^2} g_{1, n}(\o)e^{in\cdot x}\bigg).
\] 

\begin{theorem}\label{THM:GWP}
Fix $m\in\N$ and let $\al\in\left(\frac56,1\right)$ if $m=1$ and $\al\in\left(1-\frac{1}{2m^2+2m+1},1\right)$ if $m\geq 2$.
Then  \eqref{WNLW2} is almost surely globally well-posed in $C_tH^{-(1-\al)-\eps}_x$, for any $\eps>0$, with respect to the renormalized Gibbs measure $P^{2m+2}_{\al,2}$. 
Moreover, $P^{2m+2}_{\al,2}$ is invariant under the dynamics of \eqref{WNLW2}.
\end{theorem}

Note that the Wick ordered fNLW \eqref{WNLW2}
was only defined for $u$ of the form $u=S(t)(\phi_0^\o,\phi_1^\o)+w=z^\o+w$, for some $w$ of positive regularity. 
Therefore, in Theorem \ref{THM:GWP}, we use Duhamel's formula 
to decompose solutions in the above-mentioned form 
and study the solution theory for the perturbed Wick ordered fNLW that the residual term $w$
satisfies:
\begin{align}
\begin{cases}
\dt^2 w+(1 -\Dl)^\alpha  w \, +  :\! (z^\o+w)^{2m+1} \!:\, = 0 \\
(w, \dt w) |_{t = 0} = (0, 0).
\label{wick-NLW}
\end{cases}
\end{align}

\noi
The strategy of decomposing the solution of a nonlinear equation into a linear (rough) term and a residual (regular) term, that we call here `the first order expansion', is known in the realm of stochastic PDEs under the name of Da Prato-Debussche trick, see \cite{DPD}. We refer the reader to the early works by McKean \cite{McKean} and Bourgain \cite{Bou-CMP96} where this idea firstly appeared. See also \cite{BurqTz}.

\begin{remark}\rm
The proof of  Theorem \ref{THM:GWP}  is a combination of an almost sure local well-posedness theory (for \eqref{wick-NLW}) and a globalisation argument which goes back to Bourgain, see \cite{Bou-CMP94, Bou-CMP96}. The almost sure local well-posedness  is precisely the content of Proposition \ref{LWP-Str} 
in Section \ref{sec:LWP}. The extension of solutions from local-in-time to global-in-time follows by Bourgain's invariant measure and approximation argument \cite{Bou-CMP94, Bou-CMP96}. Briefly, it follows by an approximation argument which uses the fact that Proposition \ref{LWP-Str} remains valid for a truncated problem 
with uniform bounds in $N$, and the invariance of  the truncated  Gibbs measure defined by $\Pkn:=Z_N^{-1}R_N(u)\,d\mu_\al$.

We will not show details of the application of Bourgain's invariant measure and approximation argument in our context, as this is a fairly standard procedure, applicable to a wide range of equations.
We refer the readers to \cite{ROT} for a detailed exposition of the application of Bourgain's invariant measure and approximation argument in the case of nonlinear wave equations with
random data and/or stochastic forcing on a two-dimensional compact Riemannian manifold without boundary. 
\end{remark}

\begin{remark}
{\rm
For the almost sure local well-posedness theory of \eqref{wick-NLW}, we employ a fixed-point argument to exhibit the existence of a unique solution $w$ in a conveniently chosen Stricharz space $X^s(T)$ (see \eqref{ch5:cond0} and \eqref{def:str} for a definition of $X^s(T)$). 
The crucial ingredients for closing this fixed point argument are the Strichartz estimates in Lemma \ref{lemma-str}
and the considerations on the regularity of the stochastic objects $:\! z^{\l} \!:$ in Proposition \ref{PROP:sto1}, which is due to Oh and Zine \cite{Oh-Zine2022}.
}
\end{remark}

\begin{remark}\rm
In the case $m=1$, Theorem \ref{THM:GWP} improves a result from \cite{STX} for the Wick ordered cubic fNLW. 
Namely, in \cite{STX}, the authors proved almost sure global well-posedness of the Wick ordered cubic fNLW in the range $\al\in (\frac 89, 1)$, while in Theorem \ref{THM:GWP} we extend the range of $\al$ to $\al\in (\frac 56, 1)$.
\end{remark}

\begin{remark}\rm In \cite{OT2018}, Oh and Thomann consider the Wick ordered defocusing nonlinear wave equation with odd power nonlinearity, namely \eqref{NLW} with $\al=1$, construct the Gibbs measures, and prove an almost sure global well-posedness theory exploiting the first order expansion. 
The main challenge in our case compared to \cite{OT2018} is represented by the fact that the fNLW is only $\al$-smoothing, as clear from \eqref{duha-oper}, while  the wave equation is 1-smoothing. Furthermore, the regularity of the stochastic objects $:\! z^{\l} \!:$, $\l\in\{0,\dots, 2m+1\}$ 
is of order $-\l(1-\al)-\eps$, $\eps>0$ (see Proposition \ref{PROP:sto1}) in our case. 
It is this regularity that is responsible for the restriction on $\al$ in Theorem \ref{THM:GWP} or, equivalently, on the degree of the nonlinearity $u^{2m+1}$ that we can handle. 
In \cite{OT2018}, the stochastic objects have regularity $-\eps$ for any $m$, which in turn does not imply any restriction of the nonlinearity. 
\end{remark}

\subsubsection{Second order expansion}
In the proof of Theorem  \ref{THM:GWP}, more precisely in the proof of Proposition \ref{LWP-Str},
we see that the roughest term in Duhamel's formula for $w$ is 
\begin{equation}\label{eq:z_2}
z_2 : = \I(:\! z^{2m+1}\! :)=\int_0^t \frac{\sin ((t - t') \jb{\nb}^\al)}{\jb{\nb}^\al} 
:\! z^{2m+1} \!:  dt'.
\end{equation}

\noindent
Aiming at widening the range of the fractional exponent $\al$, we further decompose $w=z_2+w_2$, in the spirit of the first order expansion. 
We thus arrive at the second order expansion $u=z+z_2+w_2$, where $w_2$ solves the equation
\begin{align}
\begin{cases}
\dt^2 w_2+(1 -\Dl)^\alpha  w_2 + \, \sum_{\l=0}^{2m} \binom{2m+1}{\l} :\! z^{\l} \!:\,(z_2+w_2)^{2m+1-\l} = 0 \\
(w_2, \dt w_2) |_{t = 0} = (0, 0).
\end{cases}
\label{wick2-NLW}
\end{align}
Since, the nonlinearity in this equation is more regular than that in \eqref{wick-NLW}, 
in the case $m\geq 2$, we obtain the following improvement of Theorem  \ref{THM:GWP}.
\begin{theorem}\label{THM:GWP:sec-ord}
Fix $m\in\N$, $m\geq 2$ and $\al\in\left(1-\frac{1}{2m^2+m+1},1\right)$.
Then  \eqref{WNLW2} is almost surely globally well-posed in $C_tH^{-(1-\al)-\eps}_x$, for any $\eps>0$, with respect to the renormalized Gibbs measure $P^{2m+2}_{\al,2}$. 
Moreover, $P^{2m+2}_{\al,2}$ is invariant under the dynamics of \eqref{WNLW2}.

\end{theorem}

Theorem \ref{THM:GWP:sec-ord} is the main result of this work. 
The strategy of the proof of Theorem \ref{THM:GWP:sec-ord} is analogous to the one for Theorem \ref{THM:GWP}. 
Namely, the proof of Theorem \ref{THM:GWP:sec-ord} relies on a local well-posedness result that takes advantage of the second order expansion, namely Proposition \ref{LWP-Str-sec} in Section \ref{lwp-sec}, and on Bourgain's invariant measure and approximation argument. 
The additional difficulty in this case is that one needs 
estimates on products of stochastic objects $:\!z^{k_1}\!: z_2^{k_2}$, as depicted in Proposition \ref{PROP:sto1} (iii).
%

\begin{remark}\rm
Theorem \ref{THM:GWP:sec-ord} improves a second order expansion local well-posedness result for the Wick ordered fNLW from \cite{Oh-Zine2022}.
Both in \cite{STX}  and in \cite{Oh-Zine2022}, the authors use Sobolev embeddings when considering the local well-posedness of the Wick ordered fNLW. Our improved range of $\al$ comes from using the Strichartz estimates in Lemma \ref{lemma-str}.

 For readers' convenience, we remark here that the corresponding result of Oh and Zine is contained in \cite{Oh-Zine2022}*{Theorem 3.4} with
$d=2$, $\beta=\al$, $k=2m+1$, $m\geq1$. With these conditions, the range of $\al$ in  \cite[Theorem 3.4]{Oh-Zine2022} becomes  $\al>\max\left(1-\frac{1}{4m+3},1-\frac{1}{4m^2+1}\right)$. This is more restrictive than the condition on $\al$ in Theorem \ref{THM:GWP:sec-ord}.
\end{remark}

\begin{remark}\rm
A local well-posedness result relying on a second order expansion was also given for the Wick ordered defocusing cubic NLW  posed on the three-dimensional torus $\T^3$ in \cite{OPT}.  In \cite{OPT}, Oh, Tzvetkov and the second author only needed to control the product of stochastic objects  $ :\!z^2\!:  \I(:\!z^3\!:)$, while here we treat general nonlinearities, so we need the regularity results for general products of stochastic objects in Proposition \ref{PROP:sto1}.
\end{remark}

\begin{remark}\rm
The aim of this article and, in particular, of Theorem \ref{THM:GWP} is to discuss the almost sure global well-posedness for {\it all} the (renormalized) fNLW with an odd-power nonlinearity within a uniform framework. For this reason, we do not use the multilinear smoothing effect of the equation, which requires an analysis that is specific to the degree of the nonlinearity. Indeed, if one was to focus on a single nonlinearity (and the cubic nonlinearity would be the easiest to work with), then one could take advantage of the multilinear smoothing and, as a result, improve the range of the dispersion exponent $\alpha$ for which almost sure global well-posedness of fNLW holds for the case of that specific nonlinearity. However, as mentioned above, this is not the scope of this article.
We refer the readers to the work Bringman \cite{BringmanII} and the work of Oh, Wang and Zine \cite{OWZ} on cubic nonlinear wave equations on $\T^3$, that both take advantage of the multilinear smoothing. 
%

Another approach that has played a crucial role in the well-posedness study of the
nonlinear wave equation is the paracontrolled calculus approach, see \cite{GIP, GKO2, OOT, OOT2, BringmanI, BringmanII, BDNY}. This approach also requires an analysis that is specific to the degree of the nonlinearity and so, for the same reasons as above, we do not use it in this article. 
\end{remark}

\begin{remark}
\rm
We do not expect that a third (or higher) order expansion would improve 
Theorem \ref{THM:GWP:sec-ord}. This is due to the fact that the least regular terms in equation
\eqref{wick2-NLW} are $:\!z^{2m}\!: w_2$ and $:\!z^{2m}\!: z_2$ and a higher order expansion would not remove the term $:\!z^{2m}\!:w_2$, since it involves $w_2$.

\end{remark}

\subsection{Organisation of the paper}

The remaining part of this article is organized as
follows. In the next section, we state deterministic and stochastic tools needed for our
analysis. In Sections 3, we discuss the construction of the renormalized Gibbs measure. 
In Section 4, we discuss the almost sure local well-posedness of the renormalized fNLW that employs the first order expansion (this result is summarized in Proposition \ref{LWP-Str}).
In Section 5, we prove Proposition \ref{LWP-Str-sec}, namely the almost sure local well-posedness result for the renormalized fNLW that employs the second order expansion.
Finally, Appendix A contains a proof of Proposition \ref{LWP-Str} (which is a byproduct of the proof of Proposition \ref{LWP-Str-sec}).

\section{Preliminary deterministic and stochastic tools}

This section is devoted to the statement of useful estimates used along the paper. In particular, Subsection \ref{subsec2.1} concerns deterministic estimates, while Subsection \ref{subsec2.2} contains stochastic tools.

\subsection{Deterministic estimates}\label{subsec2.1}
We start by recalling the following estimates. They can be stated in arbitrary dimension, though we specialize to the 2D case. In what follows, we use the standard notation for the fractional Sobolev spaces $W^{s,p}(\T^2)$. $H^s(\T^2)$ corresponds to $p=2$.
\begin{lemma}[Fractional Leibniz rule]\label{LEM:prod}
\textup{(i)}  Let $s>0$,
$1\leq p_j,q_j \leq \infty$,  $j = 1, 2$, and $\frac 12\leq r\leq \infty$ such that $\frac1{p_j} + \frac1{q_j}= \frac1r$. 
Then, the following holds 
\begin{equation}  
\| \jb{\nb}^s (fg) \|_{L^r} 
\les \Big( \| f \|_{L^{p_1}} 
\| \jb{\nb}^s g \|_{L^{q_1}} + \| \jb{\nb}^s f \|_{L^{p_2}} 
\|  g \|_{L^{q_2}}\Big)
\label{bilinear+}
\end{equation}  
for any $f\in L^{p_1}(\T^2)\cap W^{s,p_2}(\T^2)$ and $g\in L^{q_2}(\T^2)\cap W^{s,q_1}(\T^2)$.\smallskip

\noi \textup{(ii)}  Let $s>0$, $1\leq p,q,r \leq \infty$ satisfying $\frac1p+\frac1q\leq\frac1r+\frac s2 $
and $q, r'\geq p'$, where $p'$ and $r'$ denote the H\"older conjugate exponents of $p$ and $r$ respectively.
Then we have
\begin{equation}  
\| \jb{\nb}^{-s}(fg) \|_{L^r} 
\les
\| \jb{\nb}^{-s} f \|_{L^p} \| \jb{\nb}^s g \|_{L^q}
\label{bil++}
\end{equation}  	
for any $f\in W^{-s,p}(\T^2)$ and $g\in W^{s,q}(\T^2)$.
\end{lemma}

\noi
See \cite{BOZ} for a proof. See also \cite{GKO, GKOT}.

The following is a classical interpolation inequality. 

\begin{lemma}\label{LEM:int}
For  $0 < s_1  < s_2$,  the following interpolation holds
\begin{equation*}
\| u \|_{H^{s_1}} \leq \| u \|_{H^{s_2}}^{\frac{s_1}{s_2}} \| u \|_{L^2}^{\frac{s_2-s_1}{s_2}}
\end{equation*}
for any $u\in H^{s_2}$.
\end{lemma}

We proceed  by recalling the following estimate on discrete convolution. See, for example,  Lemma 4.1 in \cite{MWX}.

\begin{lemma}\label{LEM:SUM}
Let  $\al, \be \in \R$ satisfy

\[ \al+ \be > 2  \qquad \text{and}\qquad \al, \be < 2.\]
\noi
Then, the following holds
\[
 \sum_{\substack{n_1,n_2\in \Z^2\\n = n_1 + n_2}} \frac{1}{\jb{n_1}^\al \jb{n_2}^\be}
\les \jb{n}^{2 - \al - \be}\]

\noi
for any $n \in \Z^2$.
\end{lemma}


We  now state the Strichartz estimates used to prove the local well-posedness results. 
For this purpose, let us introduce some notations.  We call a pair  $(p,q)$ fractional admissible provided that 
\begin{equation}\label{frac-ad}
(p,q)\in [2,\infty]\times [2,\infty], \quad (p,q)\neq(2,\infty),  \,\hbox{ and } \,\frac2p+\frac2q\leq 1.
\end{equation}
For a fixed $\al$, we define the exponent 
\begin{equation}\label{def-gam}
\gamma_{p,q}=1-\frac2q-\frac\al p.
\end{equation}
Let us consider the Cauchy problem 
\begin{align}
\begin{cases}
\dt^2 v +(1 -\Dl)^\alpha  v  = F \\
(v, \dt v) |_{t = 0} = (v_0, v_1), 
\end{cases}
\qquad (t,x) \in I \times \T^2,
\label{NLW-str}
\end{align}
for $I\subset\R$ an interval containing $0$. We have the following Strichartz estimates; see \cite[Corollary 1.4]{VDD-JDE}. 

\begin{lemma}\label{lemma-str} Let $I$ be a bounded interval containing $0$,  $\al\in(0,1)$,  $(p,q)$ be a fractional admissible pair,  $(v_0,v_1)\in H^{\gamma_{p,q}} \times H^{\gamma_{p,q}-\al}$, and let $v$ be a weak solution to \eqref{NLW-str}. The following estimate holds true:
\begin{equation}\label{eq:stri-est}
\|v\|_{L^p_t(I; L^q_x)}\lesssim \|v_0\|_{H^{\gamma_{p,q}}_x} +\|v_1\|_{H^{\gamma_{p,q}-\al}_x}+\|F\|_{L^1_t(I;H^{\gamma_{p,q}-\al}_x)}.
\end{equation}
\end{lemma}

Note that the Strichartz estimates on $\T^2$ follow from the Strichartz estimates on $\R^2$  by exploiting the finite speed of propagation  for the linear fractional wave equation. Beside \cite{VDD-JDE}, we also mention the recent paper \cite{STX} which includes a self-contained proof.

\subsection{Regularity of stochastic objects }\label{subsec2.2}

Firstly, we  introduce  basic definitions from stochastic analysis, see  the monographs \cite{Bog, Shige}, aiming at recalling the Wiener chaos estimate (see Lemma~\ref{LEM:hyp} below). 

Let $\mu$ be a Gaussian measure on a separable Banach space $B$,
with $H \subset B$ as its Cameron-Martin space. We define the abstract Wiener space  as the triple $(H, B, \mu)$. Given a complete orthonormal system $\{e_j \}_{ j \in \N} \subset B^*$ of $H^* = H$ (here $^*$ stands for the dual space),
we  define a polynomial chaos of order
$k$ to be an element of the form $\prod_{j = 1}^\infty H_{k_j}(\jb{x, e_j})$, 
where $x \in B$, $k_j \ne 0$ for only finitely many $j$'s, $k= \sum_{j = 1}^\infty k_j$, 
$H_{k_j}$ is the Hermite polynomial of degree $k_j$, 
and $\jb{\cdot, \cdot} = \vphantom{|}_B \jb{\cdot, \cdot}_{B^*}$ denotes the $B$--$B^*$ duality pairing.
The closure  of 
polynomial chaoses of order $k$ 
under $L^2(B, \mu)$ is then denoted by $\mathcal{H}_k$.
The elements in $\mathcal{H}_k$ 
are called homogeneous Wiener chaoses of order $k$.
We also set
\begin{align}
\mathcal{H}_{\leq k} := \bigoplus_{j = 0}^k \mathcal{H}_j
\notag
\end{align}

\noi
for $k \in \N$. The  Wiener chaos estimate as following.
\begin{lemma}\label{LEM:hyp}
Let $k \in \N$.
Then, we have
\begin{equation*}
\|X \|_{L^p(\mu)} \leq (p-1)^\frac{k}{2} \|X\|_{L^2(\mu)}
\end{equation*}
	
\noi
for any $p \geq 2$
and any $X \in \mathcal{H}_{\leq k}$.
	
\end{lemma}
The result above comes from the 
 the  hypercontractivity of the Ornstein-Uhlenbeck
semigroup $U(t) = e^{tL}$ due to Nelson \cite{Nelson}, see
\cite[Theorem~I.22]{Simon}, where  $L = \Dl -x \cdot \nabla$ is 
the Ornstein-Uhlenbeck operator. Here, to simplify the exposition, 
we just gave the definition of the Ornstein-Uhlenbeck operator $L$
when $B$ is the $d$-dimensional Euclidean space, see also \cite{BenOh} and \cite[Section 3]{Tz2010}. 
Let us also recall that  any element in $\mathcal H_k$ 
is an eigenfunction of $L$ with corresponding eigenvalue $-k$.

\medskip

Next we recall a lemma which will be used later on, in the proof of Lemma \ref{LEM:Dr}. 
Additionally, (a more extensive version of) this lemma
is crucial for the proof of 
Lemma \ref{lem:reg} below, a proof that we omit.
We start with the following definition. 


\begin{definition}
Given a  stochastic process $X:\R_+\mapsto \mathcal D^\prime(\T^2)$, we say it is spatially homogeneous if $\{X(t, \cdot)\}_{t\in\R^+}$ and $\{X(t, x_0+\cdot)\}_{t\in\R^+}$  have the same law for any $x_0\in\T^2$. 
\end{definition}



\begin{lemma}\label{lem:reg} Let $X:\R_+\mapsto \mathcal D^\prime(\T^2)$ be a spatially homogeneous stochastic process.
Suppose that there exists $k\in\N$ such that $X(t)$ belong to $\mathcal H_k$ for any $t\in\R_+$. 

Let $t\in\R_+$. Suppose that there exists $s_0\in\R$ such that 
\[
\E (|\ft X(t, n)|^2)\lesssim \jb{n}^{-2-2s_0}
\]
for any $n\in\Z^2$. Then, for any $s<s_0$ and any $p\geq 1$, there exists $C=C(s_0-s)$ such that
\begin{equation*}
\E(\|X(t)\|^p_{W^{s,\infty}})\leq C p^{\frac{kp}{2}}
\end{equation*}
and so $X(t)\in W^{s,\infty}$ almost surely. 

%
\end{lemma}

Lemma \ref{lem:reg}   is proved by 
straightforwardly employing  the  Wiener chaos estimate, i.e. Lemma~\ref{LEM:hyp} above. We refer to  \cite[Proposition 5]{MWX} and  \cite[Appendix A]{OOTz} for the proof. 

\medskip
We now state the main result on the regularity of stochastic objects that will be used in the proof of the almost sure local-wellposedness of the Wick ordered fNLW \eqref{WNLW2}. 
This result is Proposition 3.1 from \cite{Oh-Zine2022} in the special case of dimension $d=2$ and $\beta=\alpha$ (with $\beta$ defined in \cite{Oh-Zine2022}).

\begin{proposition}[\cite{Oh-Zine2022}]\label{PROP:sto1}

Let $\al > 0$ and $k \in \N$.\rule[-2mm]{0mm}{0mm}
\\
\noi
\textup{(i)}
Given  $\l \in \N$, let   $\al > 1-\frac{1}{\l}$.
Then,  for  $s < \l(\al - 1)$, 
$\{:\!z_N^\l\!: \}_{N \in \N}$ is 
 a Cauchy sequence
in 
$C_t(\R_+;W_x^{s,\infty}(\T^2))$,  almost surely.
In particular,
denoting the limit by $:\!z^\l \!: $,  
we have
  \[:\!z^\l \!: \, \in C_t(\R_+;W_x^{-\l(1-\al)-\eps, \infty}(\T^2))
  \]
  
  \noi 
for any $\eps >0$, almost surely.

\smallskip

\noi
\textup{(ii)}
Assume $\al > 1-\frac{1}{k}$.
Then, for   $s < k(\al - 1)+\al$, 
$\{  \I(:\!z_N^k\!:)  \}_{N \in \N}$ is 
 a Cauchy sequence
in 
$C_t(\R^+;W_x^{s,\infty}(\T^2))$,  almost surely.
In particular,
denoting the limit by $ \I(:\!z^k\!:) $, 
we have
  \[ \I(:\!z^k\!:)  \in C_t(\R_+;W_x^{-k(1-\al)+\al-\eps, \infty}(\T^2))
  \]
  
  \noi 
for any $\eps >0$, almost surely.

Furthermore, suppose that 
$\al>1-\frac{1}{k+1}$.
Let $k_2 \in \N$.
Then, for   $s < k(\al - 1)+\al$, \\
$\big\{  \big(\I(:\!z_N^k\!:)  \big)^{k_2} \big\}_{N \in \N}$ is 
 a Cauchy sequence
in 
$C_t(\R_+;W_x^{s,\infty}(\T^2))$,  almost surely.
In particular,
denoting the limit by $\big( \I(:\!z^k\!:) \big)^{k_2}$, 
we have
  \[ \big(\I(:\!z^k\!:) \big)^{k_2} \in C_t(\R_+;W_x^{-k(1-\al)+\al-\eps, \infty}(\T^2))
  \]
  
  \noi 
for any $\eps >0$, almost surely.

\smallskip
\noi
\textup{(iii)} 
Fix  integers $0 \le k_1 \le k-1$ and $0 \le k_2 \le k$. Assume that
\begin{align}
\al > 1 - \frac{1}{2k+1}. 
\label{Zx1}
\end{align}
\noi
Given $N \in \N$, define $Y_N$ by 
\[Y_N =\, :\!z_N^{k_1}\!:  \big(\I(:\!z_N^k\!:)\big)^{k_2}.\]

\noi
Then,  for  $s <- k_1(1-\al)$, 
$\{Y_N \}_{N \in \N}$ is 
 a Cauchy sequence
in 
$C_t(\R_+;W_x^{s,\infty}(\T^2))$,  almost surely.
In particular,
denoting its limit by $ 
Y =\, :\!z^{k_1}\!:  \big(\I(:\!z^k\!:)\big)^{k_2} $, 
we have
  \[ Y =\, :\!z^{k_1}\!:  \big(\I(:\!z^k\!:)\big)^{k_2}  \in C_t(\R_+;W_x^{-k_1(1-\al)-\eps, \infty}(\T^2))
  \]
  
  \noi 
for any $\eps >0$, almost surely.
\end{proposition}

\begin{remark}\rm
Let us note that the condition \eqref{Zx1} implies  the conditions in \textup{(i)} (namely $\al>1-\frac1k$) and \textup{(ii)} (namely $\al>1-\frac{1}{k+1}$). 
\end{remark}

\section{Construction of the renormalized Gibbs measures} 
 
In this section we construct the renormalized Gibbs measure associated to \eqref{WNLW2}. We first prove Proposition \ref{prop:conv:G}, then we prove the exponential integrability claimed in Theorem \ref{thm:exp:int} by means of the variational  approach of Barashkov and Gubinelli, see \cite{BG}.

\subsection{Convergence of random variables} We first prove the convergence of the random variables $G_N(u)= \int_{\T^2} :\! (\P_N u)^{2m+2}\! :  dx$ as in Proposition \ref{prop:conv:G}. Let us note that it is enough to prove the result for $p=2$, the general $p$ follows by using the Wiener chaos estimate in Lemma \ref{LEM:hyp}.
Hence, we prove  \eqref{eq:cauchy:G} for $p=2$.\medskip

With this aim in mind, we introduce some notations, and some preliminary results. 
Let $\s_{\al,N}$ be as in \eqref{G6}.
For a fixed $x \in \T^2$ and 
 $N \in \N$, 
 we also define
\begin{align}
\eta_{\al,N}(x) (\cdot) & := \frac{1}{\s_{\al,N}^\frac{1}{2}}
\sum_{|n| \leq N} \frac{\cj{e_n(x)}}{\jb{n}^\al}e_n(\cdot)
\qquad \text{and}\qquad 
\g_{\al,N} (\cdot) :=
\sum_{ |n| \leq N} \frac{1}{\jb{n}^{2\al}}e_n(\cdot),
\label{W3}
\end{align}
	
\noi
where $e_n(y) = e^{in\cdot y}$.
Note that $\eta_{\al,N}(x)(\cdot)$ is real-valued
with unitary $L^2(\T^2)$-norm, namely 
\[
\| \eta_{\al,N}(x)\|_{L^2} = 1
\]
for all  $x \in \T^2$ and all $N \in \N$.
Moreover, we have 
\begin{align}
\jb{\eta_{\al,M}(x), \eta_{\al,N}(y)}_{L^2(\T^2)}
= \frac{1}{\s_{\al,M}^\frac{1}{2}\s_{\al,N}^\frac{1}{2}} \g_{\al,N}(y-x)
= \frac{1}{\s_{\al,M}^\frac{1}{2}\s_{\al,N}^\frac{1}{2}} \g_{\al,N}(x-y), 
\label{W4}
\end{align}

\noi
for fixed $x, y\in \T^2$
and  $M\geq N$. \medskip

Furthermore, we introduce the white noise functional.
Let  $\xi(x;\o)$ be the (real-valued) mean-zero Gaussian white noise on $\T^2$
defined by
\[ \xi(x;\o) = \sum_{n\in \Z^2} g_n(\o) e^{in\cdot x},\]

\noi
where 
$\{g_{ n} \}_{n \in \Z^2}$
is a sequence of independent standard complex-valued Gaussian
random variables
conditioned that $g_{ -n} = \cj{g}_{n}$, $n \in \Z^2$.
It is easy to see that $\xi \in \H^s(\T^2) \setminus \H^{-1}(\T^2)$, $s < -1$, almost surely.
In particular, $\xi$ is a distribution, acting  on smooth functions.
In fact, the action of $\xi$ can be defined on $L^2(\T^2)$.
Namely, we define the white noise functional  $W_{(\cdot)}: L^2(\T^2) \to L^2(\O)$
 by 
\begin{equation}
 W_f (\o) = \jb{f, \xi(\o)}_{L^2} = \sum_{n \in \Z^2} \ft f(n) \cj{g}_n(\o)
\label{W0}
 \end{equation}

\noi
for a real-valued function $f \in L^2(\T^2)$.
Note that $W_f=\xi(f)$ is basically the Wiener integral of $f$.
In particular, 
$W_f$ is a real-valued Gaussian random variable
with mean 0 and variance $\|f\|_{L^2}^2$.
Moreover, 
$W_{(\cdot)}$ is unitary:
\begin{align}
 \E\big[ W_f W_h ] = \jb{f, h}_{L^2_x}
 \label{W0a}
\end{align}

\noi
for $f, h \in L^2(\T^2)$.
More generally, we have the following; see \cite{ORT}.
\begin{lemma}\label{LEM:W1}
Let $f, h \in L^2(\T^2)$ such that $\|f\|_{L^2} = \|h\|_{L^2} = 1$.
Then, for $k, m \in \N\cup\{0\}$, we have 
\begin{align*}
\E\big[ H_k(W_f)H_m(W_h)\big]
=  \dl_{km} k! [\jb{f, h}_{L^2(\T^2)}]^k.
\end{align*}
Here, $\dl_{km}$ denotes the Kronecker's delta function.
\end{lemma}

For the sake of conciseness, let us denote $u_N=\P_N u$. By using the definitions, we have 
\begin{align}
u_N (x) = \s_{\al,N}^{1/2} \frac{u_N(x)}{\s_{\al,N}^{1/2}}
= \s_{\al,N}^{1/2}W_{\eta_{\al,N}(x)},
\label{PStr4}
\end{align}
and by using the the definition of the Wick ordering and the property \eqref{eq:hermite:hom}, we easily get from \eqref{PStr4} that
\begin{equation}\label{eq:uN:hermite}
 :\! u_N^{2m+2}(x)\! : \,=H_{2m+2}(\s_{\al,N}^{1/2}W_{\eta_{\al,N}(x)}, \s_{\al,N})=\s_{\al,N}^{m+1}H_{2m+2}(W_{\eta_{\al,N}(x)})
\end{equation}

\noi We are in position to prove Proposition  \ref{prop:conv:G}  for $p=2$.
\begin{proof}[Proof of Proposition  \ref{prop:conv:G}] By definition of \eqref{eq:def:GN}, we write, by expanding  the square,
\begin{align}\nonumber 
\|G_M(u)-G_N(u)\|_{L^2(\mu_{\al})}^2&=\frac{1}{(2m+2)^2}\int_\O\left(\int_{\T^2}:\! u_M^{2m+2}\! : dx\right)^2 dP(\o)\\\nonumber
&-\frac{2}{(2m+2)^2}\int_\O\iint_{\T^2\times \T^2}:\! u_N^{2m+2}(x)\! : :\! u_M^{2m+2}(y)\! :dx\, dy\, dP(\o)\\\label{diff:g}
&+\frac{1}{(2m+2)^2}\int_\O\left(\int_{\T^2}:\! u_N^{2m+2}\! : dx\right)^2dP(\o).
\end{align}
At this point we write $(\int_{\T^2}\cdot\,dx)^2=(\int_{\T^2}\cdot\,dx)(\int_{\T^2}\cdot\,dy)$, we rewrite $:\! u_N^{2m+2}\! :$ using \eqref{eq:uN:hermite}, and by means of Lemma \ref{LEM:W1} we obtain
\begin{align*}
\eqref{diff:g}&=\frac{\s_{\al, M}^{2m+2}}{(2m+2)^2}\int_\O\iint_{\T^2\times \T^2}H_{2m+2}(W_{\eta_{\al,M}(x)})H_{2m+2}(W_{\eta_{\al,M}(y)})\,dx\,dy\,dP(\o)\\
&-\frac{2\s_{\al, M}^{m+1}\s_{\al, N}^{m+1}}{(2m+2)^2}\int_\O\iint_{\T^2\times \T^2}H_{2m+2}(W_{\eta_{\al,M}(x)})H_{2m+2}(W_{\eta_{\al,N}(y)})\,dx\,dy\,dP(\o)\\
&+\frac{\s_{\al, N}^{2m+2}}{(2m+2)^2}\int_\O\iint_{\T^2\times \T^2}H_{2m+2}(W_{\eta_{\al,N}(x)})H_{2m+2}(W_{\eta_{\al,N}(y)})\,dx\,dy\,dP(\o)\\
&=\frac{(2m+2)!\s_{\al, M}^{2m+2}}{(2m+2)^2}\iint_{\T^2\times\T^2}\jb{\eta_{\al,M}(x), \eta_{\al,M}(y)}^{2m+2}\,dx\,dy\\
&-\frac{(2m+2)!\s_{\al, N}^{2m+2}}{(2m+2)^2}\iint_{\T^2\times\T^2}\jb{\eta_{\al,N}(x), \eta_{\al,N}(y)}^{2m+2}\,dx\,dy\\
&=\frac{(2m+2)!}{(2m+2)^2}\iint_{\T^2\times\T^2}\g_{\al,M}^{2m+2}(x-y)-\g^{2m+2}_{\al,N}(x-y)\,dx\,dy.
\end{align*}
By using the finiteness of the measure of the toroidal domain, the bound $||a|^\beta-|b|^\beta |\leq c(\beta)|a-b|(|a|^{\beta-1}+|b|^{\beta-1})$ for $\beta>1$,  and the  H\"older's inequality with conjugate exponents $p=2m+2$ and $p'=\frac{2m+2}{2m+1}$, we estimate 
\begin{align}\notag
&\iint_{\T^2\times\T^2}\g_{\al,M}^{2m+2}(x-y)-\g^{2m+2}_{\al,N}(x-y)\,dx\,dy=c\int_{\T^2}\g_{\al,M}^{2m+2}(x)-\g^{2m+2}_{\al,N}(x)\,dx\\\notag
&\leq \tilde c(m)\int_{\T^2}(|\g_{\al,M}(x)|^{2m+1}+|\g_{\al,N}(x)|^{2m+1})(|\g_{\al,M}(x)-\g_{\al,N}(x)|)\,dx\\\label{new-opt}
&\leq \tilde c(m)\|\g_{\al,M}-\g_{\al,N}\|_{ L^{2m+2}}\left(\|\g_{\al,M}\|_{ L^{2m+2}}^{2m+1}+\|\g_{\al,N}\|_{  L^{2m+2}}^{2m+1}\right).
\end{align}
 By using the Hausdorff-Young's inequality to control the $L^{2m+2}$-norms, and the Sobolev's embedding $H^{m/(m+1)}\hookrightarrow L^{2m+2}$ to estimate the difference norm, we obtain
\begin{align*}
\eqref{new-opt}&\leq 2\tilde c(m) \left(\sum_{|n|\geq N}\frac{1}{\jb{n}^{4\al-2m/(m+1)}} \right)^{1/2} \left(\sum_{n\in\Z^2} \frac{1}{\jb{n}^{\frac{4\al(m+1)}{2m+1}}} \right)^{(2m+1)^2/(2m+2)}\\
&\lesssim N^{1-2\al+\frac{m}{m+1}},
\end{align*}
and the last bound goes to zero provided that 
\begin{equation}\label{al:restr}
\al>1-\frac{1}{2m+2}=\, : \al(m),
\end{equation}
indeed $\sum \jb{n}^{-\frac{4\al(m+1)}{2m+1}} <\infty$ for $\al> 1-\frac{1}{2m+2}$.
By summarizing the above calculations, we get that under the condition \eqref{al:restr} on $\al,$
\[
\|G_M(u)-G_N(u)\|_{L^2(\mu_{\al})}^2\lesssim \tilde c(m)N^{ 1-2\al+\frac{m}{m+1}},
\]
which in turn easily implies \eqref{eq:cauchy:G} when $p=2$.
\end{proof}


\subsection{Exponential integrability}\label{sec:variational}

This section is devoted to the proof of Theorem \ref{thm:exp:int}.
In this proof, we use the variational approach to Quantum Field Theory introduced in \cite{BG}. See also \cite{OST2020, GOTW, STX, LiWa} for other recent results using this approach.\medskip 

We begin by introducing some notations. Let $W(t)$ be a cylindrical Brownian motion in $L^2(\T^2)$.
Namely, we have
\begin{align}
W(t) = \sum_{n \in \Z^2} B_n(t) e_n,
\label{P1}
\end{align}

\noi
where  
$\{B_n\}_{n \in \Z^2}$ is a sequence of mutually independent complex-valued Brownian motions such that 
$ B_{-n}=\cj{B}_n$, $n \in \Z^2$. Here, by convention, we normalize $B_n$ such that $\text{Var}(B_n(t)) = t$. In particular, $B_0$ is  a standard real-valued Brownian motion.
Then, we define a centered Gaussian process $Y(t)$
by 
\begin{align}
Y(t)
=  \jb{\nabla}^{-\al}W(t).
\label{cent-gp}
\end{align}

\noi
Note that 
we have $\Law(Y(1)) = \mu_{\al}$, 
where $\mu_{\al}$ is the Gaussian measure in \eqref{G4}.
By setting  $Y_N =\P_NY $, 
we have   $\Law(Y_N(1)) = (\P_N)_*\mu_{\alpha}$, 
i.e.~the pushforward of $\mu_{\al}$ under $\P_N$.
In particular, 
we have  $\E [Y_N^2(1)] = \s_{\al,N}$,
where $\s_{\al,N}$ is as in~\eqref{G6}. \medskip

Next, let $\Ha$ denote the space of drifts, 
which are progressively measurable processes 
belonging to 
$L^2([0,1]; L^2(\T^2))$, $P$-almost surely. 
We now state the  Bou\'e-Dupuis variational formula \cite{BD, Ustunel}. Specifically, we refer to \cite[Theorem 7]{Ustunel}.

\begin{lemma}[\cite{BD, Ustunel}]\label{LEM:var3}
Let $Y$ be as in \eqref{cent-gp}.
Fix $N \in \N$.
Suppose that  $F:C^\infty(\T^2) \to \R$
is measurable and such that $\E\big[|F(\P_NY(1))|^p\big] < \infty$
and $\E\big[|e^{-F(\P_NY(1))}|^q \big] < \infty$ for some $1 < p, q < \infty$ with $\frac 1p + \frac 1q = 1$.
Then, we have
\begin{align}
- \log \E\Big[e^{-F(\P_N Y(1))}\Big]
= \inf_{\dr \in \mathbb H_a}
\E\bigg[ F(\P_N Y(1) + \P_N I(\dr)(1)) + \frac{1}{2} \int_0^1 \| \dr(t) \|_{L^2_x}^2\,dt \bigg], 
\label{P3}
\end{align}

\noi
where  $I(\dr)$ is  defined by 
\begin{align*}
 I(\dr)(t) = \int_0^t \jb{\nabla}^{-\al} \dr(t')\, dt'
\end{align*}
\noi
and the expectation $\E = \E_P $
is an expectation with respect to the underlying probability measure~$P $. 
\end{lemma}

In the following,
we construct a drift $\dr$ depending on $Y$,
and 
 the  Bou\'e-Dupuis variational formula 
(Lemma \ref{LEM:var3}) is suitable for this purpose
since 
an  expectation in  \eqref{P3} is taken
with respect to  the underlying probability $P $.

Before proceeding to the proof of Theorem \ref{thm:exp:int}, 
we state a lemma on the  path-wise regularity bounds  of 
$Y(1)$ and $I(\dr)(1)$.

\begin{lemma}  \label{LEM:Dr}
	
\textup{(i)} 
Given $h\in\N$, let $1>\al>1-\frac1h$. Let  $\eps > 0$. Then, given any finite $p \ge 1$, 
there exists $C=C(p,h,\eps)>0$ such that
\begin{align*}
\begin{split}
\E 
\Big[
 \|:\!Y_N(1)^h\!:\|_{W^{-h(1-\al)-\eps,\infty}}^p
\Big]
\leq C,
\end{split}
\end{align*}

\noi
uniformly in $N \in \N$.
	
\smallskip
	
\noi
\textup{(ii)} For any $\dr \in \Ha$, we have
\begin{align}
\| I(\dr)(1) \|_{H^{\al}}^2 \leq \int_0^1 \| \dr(t) \|_{L^2_x}^2\,dt.
\label{CM}
\end{align}
\end{lemma}
\begin{proof} 
(i) Part (i) is a straightforward adaptation of \cite[Lemma B3]{STX}. We include its proof here for completeness. 

Recall that the law of $Y_N(1)$ is $\sum_{ |n|\leq N} \frac{g_{1,n}^\omega}{\jb{n}^\al}e^{in\cdot x}$. Note that $:\!Y_N(1)^h\!:$ belongs the Wiener chaos  of order $h$. 
Direct computations yield 
\begin{align*}
\E \Big[
|\widehat{:\!Y_N(1)^h\!:}(n)|^2\Big]&=h!\sum_{n_1+\dots+n_h=n}\prod_{i = 1}^h \E\Big[ |\widehat{ Y_N(1)}(n_i)  |^2\Big]\\
&\lesssim \sum_{n_1+\dots+n_h=n}\prod_{i = 1}^h \frac{1}{\jb{n_i}^{2\al}}\lesssim \jb{n}^{-2+2h(1-\al)},
\end{align*}
where we iteratively used Lemma \ref{LEM:SUM} under the assumption $\al>1-\frac1h$. Lemma \ref{lem:reg} then gives the desired result. \smallskip

\noi (ii) The second point is a simple application of the definition of the Sobolev space $H^\al$, the Minkowski's inequality, and the Cauchy-Schwarz's inequality:
\[
\| I(\dr)(1) \|_{H^{\al}}=\Big\|\int_0^1\dr(t)\,dt\Big\|_{L^2}\leq \int_0^1\| \dr(t)\|_{L^2}\,dt
\leq \Big( \int_0^1 \| \dr(t)\|_{L^2}^2\,dt \Big)^{1/2},
\]
then the result follows by squaring both sides.
\end{proof}

\subsection{Proof of Theorem \ref{thm:exp:int}}

We now give the proof of Theorem \ref{thm:exp:int}. 
We  prove  the uniform exponential integrability 
via the variational formulation, i.e. by means of Lemma \ref{LEM:var3}. 
Since the argument is identical for any finite $p \geq 1$, we only present details for the case $p =1$. \smallskip

In view of the Bou\'e-Dupuis formula (Lemma \ref{LEM:var3}), 
it suffices to  establish a  lower bound on 
\begin{equation}
\W_N(\dr) = \E
\bigg[G_N(Y_N(1) + P_N I(\dr)(1)) + \frac{1}{2} \int_0^1 \| \dr(t) \|_{L^2_x}^2\,dt \bigg], 
\label{v_N0}
\end{equation}

\noi 
uniformly in $N \in \N$ and  $\dr \in \Ha$.
We set 
$\Dr_N : \, =  \P_N I(\dr)(1)$.
From \eqref{eq:def:GN} and \eqref{Herm:sum}, we have
\begin{align}\label{Y0}
G_N (Y_N + \Dr_N)   = \sum_{\l=0}^{2m+2}\begin{pmatrix}
2m+2 \\ \l
\end{pmatrix}\int_{\T^2}:\! Y_N^{2m+2-\l} \!:\Dr_N^{\l}\,dx 
\end{align}

\noi
Hence, from  \eqref{v_N0} and \eqref{Y0}, we have
\begin{align}
\begin{split}
\W_N(\dr)
&=\E
\bigg[\sum_{\l=0}^{2m+2}\begin{pmatrix}
2m+2 \\ \l
\end{pmatrix}\int_{\T^2}:\! Y_N^{2m+2-\l} \!:\Dr_N^{\l}\,dx 
+ \frac{1}{2} \int_0^1 \| \dr(t) \|_{L^2_x}^2\,dt 
\bigg].
\end{split}
\label{v_N0a}
\end{align}

\begin{remark}
\rm
For $\l=2m+2,$ the term $\int_{\T^2}\Dr_N^{2m+2}$ is non-negative and well defined. Indeed, in view of the Sobolev embedding $H^\al(\T^2)\subset L^{2m+2}(\T^2)$ and of \eqref{CM},  we have that $\|\Dr_N\|^{2m+2}_{L^{2m+2}}\lesssim \|\Dr_N\|^{2m+2}_{H^{\alpha}}\lesssim  \left(\int_0^1 \| \dr(t) \|_{L^2}^2dt\right)^{m+1}$ for any $\alpha\geq1-\frac{1}{m+1}$.
\end{remark}

\noi In the following, we first state 
a lemma, controlling the  terms appearing in \eqref{v_N0a}.

\begin{lemma} \label{LEM:Dr2:m} 
Fix $m\geq1$ and let $\l\in\{1,\dots, 2m+1\}.$ Let $\al>1-\frac{1}{2m+2}$. There exist $\eps>0$ and exponents $p_\l>1$ such that for any $\eta\ll 1$ there exists $c=c(\eta)>0$ such that the following holds:
\begin{equation}
\bigg| \int_{\T^2}  :\! Y_N^{2m+2-\l} \!:  \Dr_N^\l \, dx \bigg|\le c 
\|  :\! Y_N^{2m+2-\l} \!:  \|_{W^{-(2m+2-\l)(1-\al)-\eps,\infty}}^{p_\l} +\eta \left(\|\Dr_N\|_{H^\al}^{2}+\|\Dr_N\|_{L^{2m+2}}^{2m+2}\right).
\label{YY3bis}
\end{equation}
\end{lemma}

\begin{proof} 
Let 
\begin{equation}\label{cons0}
\eps<\min\left\{\Big(\al-\frac{m}{m+1}\Big)\cdot \frac{m+1}{m}, (2m+2)\Big(\al-\Big(1-\frac{1}{2m+2}\Big)\Big)\right\}.
\end{equation}
We start by noticing that 
\begin{equation}\label{cons1}
\bigg| \int_{\T^2}  :\! Y_N^{2m+2-\l} \!:  \Dr_N^\l \, dx \bigg|
\lesssim \| :\! Y_N^{2m+2-\l} \!:  \|_{W^{-(2m+2-\l)(1-\al)-\eps,\infty}}
\| \Theta_N^\l \|_{W^{(2m+2-\l)(1-\al)+\eps,1}}.
\end{equation}

\noi
Then, using the boundedness of the spatial domain, we bound from above $\| \Theta_N^\l \|_{W^{(2m+2-\l)(1-\al)+\eps,1}}$ by $\| \Theta_N^\l \|_{W^{(2m+2-\l)(1-\al)+\eps,1+\delta_\l}}$, where
\begin{equation}\label{eq:cons2}
\delta_\l:=\frac{(\al-\frac{m}{m+1})(m+1)(2m+2-\l)-\eps m}{\l(\al-\frac{m}{m+1})(m+1)+2m(m+1)(1-\al)+\eps m}.
\end{equation}
Notice that, in view of our choice of $\eps$ in \eqref{cons0}, $\delta_\l>0$. 

Then, using the fractional Leibniz rule in Lemma \ref{LEM:prod}, we have that
\begin{align}\label{cons3}
\| \Theta_N^\l \|_{W^{(2m+2-\l)(1-\al)+\eps,1+\delta_\l}}
\lesssim \|\Theta_N\|_{W^{(2m+2-\l)(1-\al)+\eps,\frac{(1+\delta_\l)(2m+2)}{2m+2-(\l-1)(1+\delta_\l)}}}\|\Theta_N\|_{L^{2m+2}}^{\l-1}.
\end{align}

Using the interpolation of Sobolev spaces and our choice of $\delta_\l$ in \eqref{eq:cons2}, we have that
\begin{align}\label{cons4}
 \|\Theta_N\|_{W^{(2m+2-\l)(1-\al)+\eps,\frac{(1+\delta_\l)(2m+2)}{2m+2-(\l-1)(1+\delta_\l)}}}\|
 \lesssim \|\Theta_N\|_{H^\al}^{\gamma_\l}\|\Theta_N\|_{L^{2m+2}}^{1-\gamma_\l},
\end{align}
where
\begin{align*}
\gamma_\l:=\frac{(2m+2-\l)(1-\al)+\eps}{\al}.
\end{align*}
Clearly $\gamma_\l>0$ and our choice of $\eps$ in \eqref{cons0} insures that we also have $\gamma_\l<1$.

Combining together equations \eqref{cons3} and \eqref{cons4}, we have that
\begin{align}\label{cons5}
 \| \Theta_N^\l \|_{W^{(2m+2-\l)(1-\al)+\eps,1+\delta_\l}}
 \lesssim \|\Theta_N\|_{H^\al}^{\gamma_\l}\|\Theta_N\|_{L^{2m+2}}^{\l-\gamma_\l}.
\end{align}
In turn, combining \eqref{cons1} and \eqref{cons5}, we obtain that
\begin{align}\label{cons6}
\bigg| \int_{\T^2}  :\! Y_N^{2m+2-\l} \!:  \Dr_N^\l \, dx \bigg|
\lesssim \| :\! Y_N^{2m+2-\l} \!:  \|_{W^{-(2m+2-\l)(1-\al)-\eps,\infty}}
 \|\Theta_N\|_{H^\al}^{\gamma_\l}\|\Theta_N\|_{L^{2m+2}}^{\l-\gamma_\l}.
\end{align}

Using the generalized Young's inequality, i.e. $ab\leq c(\eta)a^p+\eta b^q$ for $\eta\ll 1$ and $p, q>1$ satisfying $\frac 1p+\frac 1q=1$,
with $p_\l:=\frac{2m+2}{2m+2-\l-\gamma m}$ and $q_\l:=\frac{2m+2}{\gamma m +\l}$, we obtain
\begin{align*}
\bigg| \int_{\T^2}  :\! Y_N^{2m+2-\l} \!:  \Dr_N^\l \, dx \bigg|
\lesssim c(\eta)\| :\! Y_N^{2m+2-\l} \!:  \|_{W^{-(2m+2-\l)(1-\al)-\eps,\infty}}^{p_\l}
+\eta \|\Theta_N\|_{H^\al}^{\gamma_\l q_\l}\|\Theta_N\|_{L^{2m+2}}^{(\l-\gamma_\l)q_\l}.
\end{align*}

\noi
Notice that our choice of $\eps$ in \eqref{cons0} ensures that $q_\l>1$ (and so, also $p_\l>1$) for all $\l=1,2,\dots, 2m+1$. 

Lastly, using once again Young's inequality on the product $\|\Theta_N\|_{H^\al}^{\gamma_\l q_\l}\|\Theta_N\|_{L^{2m+2}}^{(\l-\gamma_\l)q_\l}$ with the conjugate exponents $\tilde p_\l:=\frac{2}{\gamma_\l q_\l}$ and
$\tilde q_\l:=\frac{2m+2}{(\l-\gamma_\l)q_\l}$, we conclude that 
\eqref{YY3bis} holds.
\end{proof}

We can now prove Theorem \ref{thm:exp:int}.
\begin{proof}[Proof of Theorem \ref{thm:exp:int}] 
Recalling the definition of $\W_N(\dr)$ in \eqref{v_N0a},
%
we observe that
\begin{align*} 
\W_N(\dr)
&=\E\bigg[\int_{\T^2}:\! Y_N^{2m+2} \!:dx \bigg]
+\E
\bigg[\sum_{\l=1}^{2m+1}\begin{pmatrix}
2m+2 \\ \l
\end{pmatrix}\int_{\T^2}:\! Y_N^{2m+2-\l} \!:\Dr_N^{\l}\,dx 
\bigg]\\
&+\E \bigg[ \|\Dr_N\|_{L^{2m+2}}^{2m+2}+\frac{1}{2} \int_0^1 \| \dr(t) \|_{L^2_x}^2 \,dt 
\bigg]=:\mathcal{A} +\mathcal{B}+\mathcal{C}.
\end{align*}
Note that the term $\mathcal{C}$ is positive. 
\noi 
By Proposition \ref{prop:conv:G}, for $\al>1-\frac{1}{2m+2}$, we have that the first term 
$\mathcal{A}$ is uniformly bounded in $N$, so $\mathcal{A}\gtrsim -1$. Lastly, by Lemma \ref{LEM:Dr2:m} and Lemma \ref{LEM:Dr}, for $\al>1-\frac{1}{2m+1}$, we obtain that 
\[
\mathcal{B}\gtrsim -1-\eta\, \mathcal{C}. 
\]
Therefore, for $\al>1-\frac{1}{2m+2}$, we obtain that
\[
\inf_{N, \theta} \W_N(\theta)\gtrsim -1+(1-\eta)\,\mathcal{C}\gtrsim -1.
\]
The uniform (in $N$ and $\theta$) bound above, jointly with \eqref{P3} gives the desired result. 

\end{proof}

\section{First order expansion almost sure local well-posedness}\label{sec:LWP}
In this Section, we discuss the following first result on the almost sure local well-posedness of the defocusing Wick ordered fraction NLW \eqref{WNLW2}. 

\begin{proposition}\label{LWP-Str}
{\rm
Fix $m\in\N$.
For $m=1$, let $\al\in (\frac 56, 1)$.
For $m\geq 2$, let $\al\in \left(1-\frac{1}{2m^2+2m+1}, 1\right)$. 

Then, there exists $\Omega^1\subset \Omega$ with $P(\Omega^1)=1$ and with the property that for each $\omega\in\Omega^1$ there exists $T_\omega>0$ such that 
 \eqref{WNLW2} admits a unique local solution in the class $C_t([0,T];H_x^{-(1-\al)-\eps})$, for any $0<\eps\ll1$.
 More precisely, the solution is of the form 
 \[
u=S(t)(\phi_0^\o, \phi_1^\o)+w\in C_t([0,T]; H_x^{-(1-\al)-\eps}(\T^2)) +X^s(T),\]
where $s\in (0,1)$ and $X^s(T)$ is defined below, in \eqref{def:str}.  
(See the hypotheses of Proposition \ref{LWP-Str-det} and Remark \ref{rem:LWP1} below for the precise conditions on the different indices appearing in the definition of $X^s(T)$.)
}
\end{proposition}

As discussed in the introduction, our first local well-posedness result is based on the first order expansion $u=S(t)(\phi_0^\o, \phi_1^\o)+w=z+w$, where $w$ solves \eqref{wick-NLW}.
With this approach, the almost sure local well-posedness of \eqref{WNLW2} relies on 
proving an almost sure local well-posedness result for \eqref{wick-NLW}.
In view of Duhamel's formulation for \eqref{wick-NLW}, we define the map $\Psi$ as follows: 
\begin{equation}\label{cont-map}
 \Psi(w)(t):=
 \int_0^t \frac{\sin ((t - t') \jb{\nb}^\al)}{\jb{\nb}^\al} 
  :\! (z+w)^{2m+1}\!:  dt'.
\end{equation}

\noi
We recall from the property of the Hermite polynomials \eqref{Herm:sum}, that 
\begin{align*}
 :\! (z+w)^{2m+1}\!:  \, \, 
  = \sum_{\l = 0}^{2m+1}
\begin{pmatrix}
2m +1 \\ \l
\end{pmatrix}
  :\! z^\l\!: w^{2m+1 - \l}.
\end{align*}
Then the map $ \Psi(w)(t)=\Psi^\o_{\{: z^{\l} :,\, \l\in\{1,\dots,2m+1\} \}}(w)(t)$ can be rewritten as
\begin{align}\nonumber
 \Psi(w)(t) &=  \sum_{\l = 0}^{2m+1}\begin{pmatrix}
2m +1 \\ \l
\end{pmatrix}
 \int_0^t \frac{\sin ((t - t') \jb{\nb}^\al)}{\jb{\nb}^\al} \left( :\! z^\l\!: w^{2m+1 - \l} \right)dt'.
 \end{align}

We  prove that
the Cauchy problem \eqref{wick-NLW} admits a solution by showing that the map $\Psi$ admits a fixed point in a suitable space.
%
%
The map $\Psi$ depends on the set of stochastic objects  
\begin{equation}\label{en-data-set}
\Xi^1 :\,=\{:\! z^{\l}\!: \}_{\l\in \{1,\dots,2m+1\}},
\end{equation}
and the existence of a fixed point for $\Psi$ will be shown in a deterministic way. Basically, we decompose the ill-posed (due to the rough initial data) solution-map of \eqref{WNLW2} in two steps:
\begin{equation}\label{fact-sol-map}
\begin{aligned} (u_0^\omega, u_1^\omega)&\xmapsto{{\rm lift}}
\big\{: z^{\l} :,\, \l\in\{1,\dots,2m+1\} \big\}
  \xmapsto{\Psi}  w\in X^s
 \end{aligned}
 \end{equation}
 and so
 \begin{align*} (u_0^\omega, u_1^\omega)
 \longmapsto u=z^\o+w\in C_t([0,T]; H^{-(1-\al)-\eps}_x) +X^s\subset  C_t([0,T]; H^{-(1-\al)-\eps}_x),
 \end{align*}

\noindent
where $X^s$ is a suitable Strichartz space, see \eqref{def:str} below.
In the first lifting step we generate the enhanced data set of stochastic objects defined in \eqref{en-data-set} with low regularity, and in the second we obtained the solution (in the sense specified in the introduction) via a fixed point argument in a deterministic setting.
\medskip 

For the second step, we consider the following deterministic equation:
\begin{align}
\begin{cases}
\dt^2 w+(1 -\Dl)^\alpha  w + w^{2m+1}  \, +  \sum_{\l = 1}^{2m+1}
\begin{pmatrix}
2m +1 \\ \l
\end{pmatrix}
  \Xi^1_\l w^{2m+1 - \l} = 0 \\
(w, \dt w) |_{t = 0} = (0, 0),
\label{wick-NLW-det}
\end{cases}
\end{align}

\noi
where $\Xi^1_\l\in C(\R_+;W^{-\l(1-\al)-\eps}(\T^2))$
for any $\eps>0$ and any $\l=1,2,\dots, 2m+1$.
We set
$\mathcal Z_1: =\{C([0,1];W^{-\l(1-\al)-\eps}(\T^2))\}_{\l=1}^{2m+1}$ and 
\begin{equation*}
\|\Xi^1\|_{\mathcal Z_1}:=\sum_{\l=1}^{2m+1} \|\Xi^1_\l\|_{L^{\infty}([0,1];W^{-\l(1-\al)-\eps}(\T^2))}.
\end{equation*}

\noindent
Using Duhamel's formula, we have that 
\begin{align}
w(t)=  &\int_0^t \frac{\sin ((t - t') \jb{\nb}^\al)}{\jb{\nb}^\al}  \Xi^1_{2m+1}
dt' \label{cont-map1}\\\label{cont-map2}
&+\sum_{\l = 1}^{2m}\begin{pmatrix}
2m +1 \\ \l
\end{pmatrix}
 \int_0^t \frac{\sin ((t - t') \jb{\nb}^\al)}{\jb{\nb}^\al} \left(\Xi^1_\l w^{2m+1 - \l}  \right)dt'\\\label{cont-map3}
&+ \int_0^t \frac{\sin ((t - t') \jb{\nb}^\al)}{\jb{\nb}^\al}  w^{2m+1}.
\end{align}

Let us note that the critical exponent for \eqref{NLW} is $s_{\rm crit}=1-\frac{\al}{m}$. To start with, we fix 
\begin{equation}\label{ch5:cond0}
1>s>s_{\rm crit}=1-\frac{\al}{m}.
\end{equation}
Given $0<T\leq1$, we define the space 
\begin{equation}\label{def:str}
X^s(T):=L^\infty_t([0,T]; H^s_x(\T^2))\cap L^p_t([0,T];W^{s-\gamma_{p,q},q}_x(\T^2))
\end{equation}
where $(p,q)$ is a fractional admissible pair depending on $m$ (see \eqref{frac-ad} and  \eqref{def-gam} for the definitions). 

Then we have the following (deterministic) local well-posedness results for \eqref{wick-NLW-det}.

\begin{proposition}\label{LWP-Str-det}
{\rm
Fix $m\in\N$, $p>2m+1$ and $q\geq 4m+2$. For $m=1$ let $p<5$ and $\al>\frac 56$. 
For $m\geq 2$, let $\al\in\left(1-\frac{\beta}{2m^2+2m+\beta},1\right)$,
where $\beta:=\frac 12+\frac{2m+1}{2p}$. 

Set $\tilde s_1:=(2m+2)\al-(2m+1)$. For $m=1$, let $s\in \left(\max\{2(1-\al), 1-\al\beta\}, \min\{\tilde s_1, 1-\al+\gamma_{p,q}\}\right)$.
For $m\geq 2$, let  $s\in \left(1-\frac{\al \beta}{m}, \min\{\tilde s_1, 1-\al+\gamma_{p,q}\}\right)$.

Then, given an enhanced data set $\Xi^1=(\Xi_1,\Xi_2,\dots, \Xi_{2m+1})\in\mathcal Z_1$,
there exists $T=T(\|\Xi^1\|_{\mathcal Z_1})\in (0,1]$ 
and a unique solution $w$ of \eqref{wick-NLW-det} in $X^s(T)$. 
Furthermore, the map $\Xi^1\in \mathcal Z_1\mapsto w\in X^s(T)$ is continuous.

}
\end{proposition}

\begin{remark}\label{rem:LWP1}
{\rm
Let $m\geq 2$ and $\al\in\left(1-\frac{1}{2m^2+2m+1},1\right)$. Let $q\geq 4m+2$. 
Then,
there exists $\beta=\beta (m,\al)\in(\frac 12,1)$ such that 
$\al\in\left(1-\frac{\beta}{2m^2+2m+\beta},1\right)$. After setting $p=p(m,\al):=\frac{2m+1}{2\beta-1}$, the conclusion of Proposition \ref{LWP-Str-det} still holds (for this choice of $p$ and $q$ and with $s$ and $\Xi^1$ as in Proposition \ref{LWP-Str-det}).
%
}
\end{remark}


\medskip

Proposition \ref{LWP-Str} follows from
Proposition \ref{LWP-Str-det}, Remark \ref{rem:LWP1} and  Proposition \ref{PROP:sto1} (i).
The proof of Proposition \ref{LWP-Str} (in fact, the proof of Proposition \ref{LWP-Str-det}) will be given in the appendix. Indeed, in the next section,  we give a stronger a.s.  local well-posedness result by means of a second order expansion. 


\begin{remark}\rm

Note that the uniqueness of the solution only holds in the Strichartz space $X^s(T)$, and we refer to this kind of local well-posedness result as {\it conditional} (in the remainder $w$). \smallskip
An {\it unconditional} local well-posedness theory can be done by using only Sobolev's embeddings, in the spirit of \cite{GKOT}. Nonetheless, such an approach would give a weaker result in term of the range of $\al$. 
\end{remark}

\section{Second order expansion almost sure local well-posedness}\label{lwp-sec}

We now prove  the a.s. local well-posedness of the renormalized fNLW employing the second order expansion.
Similarly to  the previous section, the functional space we are going to work with is the space $X^s(T):=L^\infty_T H^s_x\cap L^p_TW^{s-\gamma_{p,q},q}_x$ introduced in \eqref{def:str}. We recall that  $(p,q)$ are fractional admissible, see \eqref{frac-ad}, and that $\gamma_{p,q}$ is defined in \eqref{def-gam}. 
\medskip

\begin{proposition}\label{LWP-Str-sec}
{\rm
Fix $m\in\N$. For $m=1$, let $\al\in (\frac 67,1)$. 
For $m\geq 2$, let $\al\in\left(1-\frac{1}{2m^2+m+1}, 1\right)$. 

 Then, there exists $\Omega^2\subset \Omega$ with $P(\Omega^2)=1$ and with the property that for each $\omega\in\Omega^2$ there exists $T_\omega>0$ such that 
 \eqref{WNLW2} admits a unique local solution in the class $C_t([0,T_\omega];H_x^{-(1-\al)-\eps})$, for any $0<\eps\ll1$.
 More precisely, the solution is of the form 
\[
u=z+z_2+w_2\in C_t([0,T_\omega];H_x^{-(1-\al)-\eps}(\T^2))+C_t([0,T_\omega]; W^{-(2m+1)(1-\al)+\al -\eps,\infty}_x(\T^2))+ X^s(T_\omega).
\]

\noi
(See the hypotheses of Proposition \ref{LWP-Str-sec-det} and Remark \ref{LWP-Str-sec-det-rem} below for the precise conditions on the different indices appearing in the definition of $X^s(T_\o)$.)

}
\end{proposition}

\begin{remark}
\rm
In the case of the cubic fNLW, when $m=1$, Proposition \ref{LWP-Str-sec} does not provide an improvement over Proposition \ref{LWP-Str}. While we do see an improvement at the level of the deterministic well-posedness results that each of these two propositions is based on (i.e. the condition $\al>\frac 45$ in Proposition \ref{LWP-Str-sec-det} bellow is less restrictive than the condition $\al>\frac 56$ in Proposition \ref{LWP-Str-det}), in order to obtain Proposition \ref{LWP-Str-sec}, we apply Proposition \ref{PROP:sto1} (iii), 
which requires the further restriction $\al>1-\frac{1}{4m+3}=\frac 67$. Note that,
in the proof of Proposition \ref{LWP-Str}, we only use Proposition \ref{PROP:sto1} (i), which only requires $\al>1-\frac{1}{2m+1}=\frac 23$.

\end{remark}

As in the previous section, we create an enhanced data set and then solve the problem by performing a fixed point argument in a deterministic way.
In this case, we further factorize the ill-posed solution map likely to \eqref{fact-sol-map}, but by constructing a larger enhanced data set consisting of 
\[
:\!z^{\l}\!:  \big(\I(:\!z^{2m+1}\!:)\big)^{2m+1-\l-j}, \quad \l\in\{0,\dots, 2m\}, j\in\{1,\dots, 2m+1-\l\}.
\]

Recalling equation \eqref{wick2-NLW} and Duhamel's formula for this initial value problem, we
consider the following deterministic problem:

\begin{align}
w_2(t) &= \Psi(w_2)(t)\label{eq with Z sec}\\
&:=\sum_{\l=0}^{2m}\sum_{j=0}^{2m+1-\l}b_{m,\l}c_{m,\l,j}\int_0^t \frac{\sin ((t - t') \jb{\nb}^\al)}{\jb{\nb}^\al} \left( :\! Z^\l\! : (\mathcal I(:\! Z^{2m+1}\!:))^{2m+1-\l-j}w_2^j \right)
  dt'\label{cont-map:2ord}
  \\
  \label{cont-map:2ord0}
  &=\sum_{\l=0}^{2m}b_{m,\l}c_{m,\l,0}\int_0^t \frac{\sin ((t - t') \jb{\nb}^\al)}{\jb{\nb}^\al} \left( :\! Z^\l\! : (\mathcal I(:\! Z^{2m+1}\!:))^{2m+1-\l} \right) dt'\\
  &+\sum_{\l=1}^{2m}\sum_{j=1}^{2m+1-\l}b_{m,\l}c_{m,\l,j}\int_0^t \frac{\sin ((t - t') \jb{\nb}^\al)}{\jb{\nb}^\al} \left( :\! Z^\l\! : (\mathcal I(:\! Z^{2m+1}\!:))^{2m+1-\l-j}w_2^j \right) dt' \label{eq:2ord-z} \\\label{eq:2ord-a}
&+\sum_{j=1}^{2m+1}b_{m,0}c_{m,0,j}\int_0^t \frac{\sin ((t - t') \jb{\nb}^\al)}{\jb{\nb}^\al} \left(  (\mathcal I(:\! Z^{2m+1}\!:))^{2m+1-j}w_2^j \right) dt'.
\end{align}

We also introduce the following notations. 
Set 
\begin{equation*}
\Xi^2:=\left\{:\!Z^{k_1}\!:  \big(\I(:\!Z^{2m+1}\!:)\big)^{k_2}, \quad k_1\in\{0,\dots, 2m\}, k_2\in\{0,\dots, 2m+1-k_1\}\right\}
\end{equation*}
and
\begin{equation*}
\|\Xi^2\|_{\mathcal Z_2}:=\sum_{k_1=0}^{2m}\sum_{k_2=0}^{2m+1-k_1}\|:\!Z^{k_1}\!:(\mathcal I(:\!Z^{2m+1}\!:))^{k_2}\|_{C([0,1];W_x^{s_{k_1,k_2},\infty}(\T^2))},
\end{equation*}
where
\begin{align*}
s_{k_1,k_2}:=
\begin{cases}
-k_1(1-\al)-\eps, \quad {\text if} \quad k_1\geq 1,\\
-(2m+1)(1-\al)+\al-\eps, \quad {\text if} \quad k_1=0.
\end{cases}
\end{align*}

Then, Proposition \ref{LWP-Str-sec} follows from the following deterministic local well-posedness result and by applying Proposition \ref{PROP:sto1}.

\begin{proposition}\label{LWP-Str-sec-det}
{\rm
Fix $m\in\N$, $p>2m+1$ and $q\geq 4m+2$. 
For $m=1$ let $\al>\frac 45$. 
For $m\geq 2$, 
let $\al\in\left(1-\frac{\beta}{2m^2+m+\beta},1\right)$,
where $\beta:=\frac 12+\frac{2m+1}{2p}$. 

Set $s_1:=(2m+1)\al-2m$. For $m=1$, let $s\in \left(\max\{2(1-\al), 1-\al\beta\}, \min\{s_1, 1-\al+\gamma_{p,q}\}\right)$.
For $m\geq 2$, let  $s\in \left(1-\frac{\al \beta}{m}, \min\{s_1, 1-\al+\gamma_{p,q}\}\right)$.

Assume that $\Xi^2\in C([0,1];H^{-(1-\al)-\eps}(\T^2))$ is such that $\|\Xi^2\|_{\mathcal Z_2}<\infty$.

Then, there exists $T=T(\|\Xi^2\|_{\mathcal Z_2}) >0$ such that the map $\Psi$ defined in \eqref{cont-map:2ord} admits a unique fixed point in the space $X^s(T)$. 
Consequently, \eqref{eq with Z sec} is locally well-posed in $X^s(T)$. 
}
\end{proposition}

\begin{remark}\label{LWP-Str-sec-det-rem}
{\rm
Let $m\geq 2$ and $\al\in\left(1-\frac{1}{2m^2+m+1},1\right)$.  Let $q\geq 4m+2$.
Then,
there exists $\beta=\beta (m,\al)\in(\frac 12,1)$ such that 
$\al\in\left(1-\frac{\beta}{2m^2+m+\beta},1\right)$. After setting $p=p(m,\al):=\frac{2m+1}{2\beta-1}$, the conclusion of Proposition \ref{LWP-Str-sec-det} still holds (for this choice of $p$ and $q$ and with $s$ and $\Xi^2$ as in Proposition \ref{LWP-Str-sec-det}).
}
\end{remark}

\begin{proof}[Proof of Remark \ref{LWP-Str-sec-det-rem}]
Assume that $\al>1-\frac{1}{2m^2+m+1}$. Then, there exists $0<\eta\ll 1$, such that 
$\al>1-\frac{1}{2m^2+m+1}+\eta $. Then, there exists $\beta=\beta (m,\eta)\in (\frac 12, 1)$ such that 
$1-\frac{1}{2m^2+m+1}+\eta = 1-\frac{\beta}{2m^2+m+\beta}$.
So, $\al\in \left(1-\frac{\beta}{2m^2+m+\beta}, 1\right)$ 
and, after taking $p:=\frac{2m+1}{2\beta-1}$, Proposition \ref{LWP-Str-sec-det} applies.
Notice that when $\eta \searrow 0$, then $\beta \nearrow 1$ and so $p \searrow 2m+1$.
\end{proof}

We next show how Proposition \ref{LWP-Str-sec} follows from Proposition \ref{LWP-Str-sec-det} and Remark \ref{LWP-Str-sec-det-rem} and from Proposition \ref{PROP:sto1}.

\begin{proof}[Proof of Proposition \ref{LWP-Str-sec}]
Applying Proposition \ref{PROP:sto1}(i) with $\l=0,\dots, 2m$, Proposition \ref{PROP:sto1}(ii) with $k=2m+1$ and $k_2=1,\dots, 2m+1$, and Proposition \ref{PROP:sto1}(iii) with $k=2m+1$, $k_1=1,\dots, 2m$ and $k_2=1,\dots, 2m+1-k_1$, it follows that, provided $\al>1-\frac{1}{4m+3}$, there exists a set $\tilde\Omega\subset \Omega$ with $P(\tilde\Omega)=1$ such that for all $\omega\in\tilde \Omega$ we have that 
\begin{equation*}
\|\left\{:\!(z^\o)^{k_1}\!:  \big(\I(:\!(z^\o)^{2m+1}\!:)\big)^{k_2}, \quad k_1\in\{0,\dots, 2m\}, k_2\in\{0,\dots, 2m+1-k_1\}\right\}\|_{\mathcal Z_2}<\infty.
\end{equation*}

Fix $\omega\in\tilde\Omega$. Applying Remark \ref{LWP-Str-sec-det-rem} with $\Xi^2$ replaced by $\left\{:\!(z^\o)^{k_1}\!:  \big(\I(:\!(z^\o)^{2m+1}\!:)\big)^{k_2}, \quad k_1\in\{0,\dots, 2m\}, k_2\in\{0,\dots, 2m+1-k_1\}\right\}$, the result of Proposition \ref{LWP-Str-sec} follows with 
\[T_\omega:=T\left(\left\|\left\{:\!(z^\o)^{k_1}\!:  \big(\I(:\!(z^\o)^{2m+1}\!:)\big)^{k_2}, \quad k_1\in\{0,\dots, 2m\}, k_2\in\{0,\dots, 2m+1-k_1\}\right\}\right\|_{\mathcal Z_2}\right)\]
as long as 
\begin{align*}
\al>\max\Big\{1-\frac{1}{2m^2+m+1},1-\frac{1}{4m+3}\Big\}
=
1-\frac{1}{2m^2+m+1}, \quad \text{if} \quad m\geq 2,\\
\end{align*}
and as long as $\al>\max\{\frac 45, 1-\frac{1}{4m+3}\}=\frac 67$, if $m=1$.
In particular, we note that in the case $m=1$ of a cubic nonlinearity, the condition 
$\al>1-\frac{1}{4m+3}=\frac 67$ needed in order to apply Proposition \ref{PROP:sto1}(iii)
is more restrictive than the condition $\al>\frac 45$ from Proposition \ref{LWP-Str-sec-det}.

\end{proof}

We conclude this section with the proof of Proposition \ref{LWP-Str-sec-det}.

\begin{proof}[Proof of Proposition \ref{LWP-Str-sec-det}]  We prove the  estimates needed to show that the map defined in \eqref{cont-map:2ord} is a contraction on a closed ball of $X^s(T)$ for $T=T(\|\Xi^2\|_{\mathcal Z_2})\leq 1$ sufficiently small. \smallskip

\noi $\bullet$ {\bf Case 1. Estimate on the source term \eqref{cont-map:2ord0}.}
By the Strichartz estimates in Lemma \ref{lemma-str} and the embedding of Lebesgue spaces on finite measure domains, we have
\begin{align*}
\| \eqref{cont-map:2ord0}\|_{X^s(T)}&\lesssim \sum_{\l=0}^{2m}\big\| :\!Z^\l \!: (\mathcal I(:\!Z^{2m+1}\!:))^{2m+1-\l}\|_{L^1_TH^{s-\al}_x}\\
&\lesssim T \sum_{\l=0}^{2m}\| \jb{\nabla}^{s-\al}\left(:\!Z^\l \!: (\mathcal I(:\!Z^{2m+1}\!:))^{2m+1-\l}\right)\big\|_{L^\infty_T L^\infty_x}.
\end{align*}
So, provided that $s-\al<-2m(1-\al)$, that is
\begin{equation}\label{eq:source-2ord}
s<(2m+1)\al-2m= : s_1,
\end{equation}
we have
\begin{equation}\label{eq:a}
\| \eqref{cont-map:2ord0}\|_{X^s}\lesssim T \sum_{\l=0}^{2m}\|:\!Z^\l \!: (\mathcal I(:\!Z^{2m+1}\!:))^{2m+1-\l}\|_{L^\infty_T W^{-\l(1-\al)-\eps,\infty}_x}\lesssim T\|\Xi^2\|_{\mathcal Z_2}.
\end{equation}

\noi Note that conditions \eqref{eq:source-2ord} and \eqref{ch5:cond0} imply 
\begin{equation}\label{eq:first-comp}
\al>1-\frac{1}{2m^2+m+1}
\end{equation}

In conclusion, \eqref{eq:a} with $s_{\rm crit}<s<s_1$ holds 
provided the condition \eqref{eq:first-comp} holds. 

 \medskip
 
\noi $\bullet$ {\bf Case 2. Estimates on \eqref{eq:2ord-z}.} 
Let $\eps>0$ arbitrarily small and set $s_\al:=-(1-\al)-\eps$. 
We start by choosing $\theta_\l, r, \tilde r, \tilde r_1$ satisfying the conditions below. These parameters will appear in the estimate of \eqref{eq:2ord-z}.
\begin{align*}
&\theta_\l:=1-\frac{s+\l s_\al}{\gamma_{p,q}}, \quad 0\leq \theta_\l\leq 1, \quad \textrm{for all} \quad \l=1,2,\dots, 2m, \\
&1<r=r(\l)\leq 2, \quad 1<\tilde r=\tilde r(j,\l)<\infty, \quad 2\leq \tilde r_1=\tilde r_1(\l) \leq q, \\
&\frac{1-(s-\al)}{2}= \frac 1r-\frac{\l s_\al}{2}>\frac{1}{\tilde r}= \frac{1}{\tilde r_1}+\frac{j-1}{q}= \frac{\theta_\l}{2}+\frac{1-\theta_\l}{q}+\frac{j-1}{q}, \\
&\quad \quad \quad \quad \quad \quad \textrm{for all} \quad j=1,2,\dots, 2m+1-\l, \quad  \l=1,2,\dots, 2m.
\end{align*}

\noindent
Let us justify why it is possible to choose $\theta_\l, r, \tilde r, \tilde r_1$ satisfying the above conditions, given our choice for $\al, s, p, q$. 

First, $\theta_\l\leq 1$ for any $\l=1,2,\dots, 2m$ is equivalent to $s\geq -\l s_\al =\l(1-\al)+\l\eps$ for all $\l=1,2,\dots, 2m$. For $m\geq 2$, recall that we have $s>s_{\rm crit}$ and it can be easily verified that $s_{\rm crit}> 2m(1-\al)$ for $\al>1-\frac{1}{2m^2+m+1}$. Therefore, $\theta_\l\leq 1$ holds for $m\geq 2$. For $m=1$, we need $s>2(1-\al)$ to insure that $\theta_\l\leq 1$ for $\l=1,2$. On the other hand, $\theta_\l\geq 0$ for any $\l=1,2,\dots, 2m$ is equivalent to $s\leq \gamma_{p,q}-s_\al$ which clearly holds for the choice $s<1-\al+\gamma_{p,q}$. 

Secondly, in order to choose $r,\tilde r,\tilde r_1$ satisfying the sequence of equalities/inequalities above, it suffices to have that
\begin{equation*}
\frac{1-(s-\al)}{2}>\frac{\theta_\l}{2}+\frac{1-\theta_\l}{q}+\frac{j-1}{q},
\end{equation*}
for all $j=1,2,\dots, 2m+1-\l$ and all $\l=1,2,\dots, 2m$. 
This boils down to
\begin{equation*}
\Sigma_\l:=(1-\theta_\l)\left(\frac 12-\frac 1q\right)-\frac{s-\al}{2}-\frac{2m-\l}{q}>0,
\end{equation*}
for all $\l=1,2,\dots, 2m$.
Using the definition of $\theta_\l$ and the fact that $\frac12-\frac 1q=\frac{\gamma_{p,q}}{2}+\frac{\al}{2p}$, we note that
\begin{equation*}
\Sigma_\l=\frac{\al+\l s_\al}{2}+(s+\l s_\al)\frac{\al}{2p\gamma_{p,q}}-\frac{2m-\l}{q}.
\end{equation*}
We have already seen that $s\geq -\l s_\al$ for all $\l=1,2,\dots, 2m$, so the middle term in the last expression for $\Sigma_\l$ is non-negative. So it is enough to show that
\begin{equation*}
\frac{\al+\l s_\al}{2}-\frac{2m-\l}{q}>0, \quad \textrm{for all} \quad \l=1,2,\dots, 2m. 
\end{equation*}
A calculation shows that
\begin{equation*}
\frac{\al+\l s_\al}{2}-\frac{2m-\l}{q}=\frac{\l+1}{2}\left(\al-(1-f(\l))\right)-\l\eps, 
\end{equation*}
where $f(x):=\frac{2x+q-4m}{q(x+1)}$. Since $q\geq 4m+2$, we have that $f'(x)\leq 0$ for all $x$, so $f$ is a decreasing function. In particular, $f(\l)\geq f(2m)$ for all $\l=1,2,\dots, 2m$. Therefore,
\begin{equation*}
\frac{\al+\l s_\al}{2}-\frac{2m-\l}{q}> \frac{\l+1}{2}\left(\al-(1-f(2m))\right)-\l\eps=\frac{\l+1}{2}\left(\al-\left(1-\frac{1}{2m+1}\right)\right)-\l\eps>0
\end{equation*}
for $\eps$ sufficiently small, since we have $\al>1-\frac{1}{2m^2+m+1}$ for $m\geq 2$ and $\al>\frac 45$ for $m=1$.
In conclusion, we have indeed that $\Sigma_\l>0$.

Note that the condition $0<s<s_1$ allows us to choose $\eps>0$ sufficiently small
to insure that $0\leq \l s_\al-(s-\al)<1$ holds. Then the equality that defines $r$, $\frac 1r=\frac 12+\frac{\l s_\al-(s-\al)}{2}$, implies that $r\in (1,2]$.

We now proceed to estimating \eqref{eq:2ord-z}. Using the Strichartz estimates in Lemma \ref{lemma-str} and the dual Sobolev embedding $L^r(\T^2)\subset H^{s-\al-\l s_\al}(\T^2)$, we obtain
\begin{align*}
\|\eqref{eq:2ord-z}\|_{X^s}&\lesssim \sum_{\l=1}^{2m}\sum_{j=1}^{2m+1-\l} \|:\!Z^{\l}\!:(\mathcal I(:\!Z^{2m+1}\!:))^{2m+1-\l-j} w_2^j\|_{L^1_tH^{s-\al}_x}\\
&=\sum_{\l=1}^{2m}\sum_{j=1}^{2m+1-\l} \|\jb{\nabla}^{\l s_\alpha}\big(:\!Z^{\l}\!:(\mathcal I(:\!Z^{2m+1}\!:))^{2m+1-\l-j} w_2^j\big)\|_{L^1_tH^{s-\al-\l s_\al}_x}\\
&\lesssim \sum_{\l=1}^{2m}\sum_{j=1}^{2m+1-\l} \|\jb{\nabla}^{\l s_\al}\left( :\!Z^{\l}\!:(\mathcal I(:\!Z^{2m+1}\!:))^{2m+1-\l-j} w_2^j\right)\|_{L^1_t L^{r}_x}
\end{align*}
As noted above, the condition $s<s_1$ allows us to choose $\eps>0$ sufficiently small
to insure that $s-\al-\l s_\al\leq 0$ holds for all $\ell=1,2,\dots, 2m$. This is needed for applying the dual Sobolev embedding.

Then, by Lemma \ref{LEM:prod} (ii) with $1<\tilde r<\infty$ such that $\frac{1}{\tilde r}<\frac 1r-\frac{\l s_\al}{2}$, 
we have:
\begin{align*}
\|\eqref{eq:2ord-z}\|_{X^s}
&\lesssim \sum_{\l=1}^{2m}\sum_{j=1}^{2m+1-\l} \|\jb{\nabla}^{\l s_\al}\left( :\!Z^{\l}\!:(\mathcal I(:\!Z^{2m+1}\!:))^{2m+1-\l-j}\right)\|_{L^{\infty}_TL^\infty_x} \|\jb{\nabla}^{-\l s_\al}\left(w_2^j\right)\|_{L^1_t L^{\tilde r}_x}\\
& \lesssim \|\Xi^2\|_{\mathcal Z_2} \sum_{\l=1}^{2m}\sum_{j=1}^{2m+1-\l} \|\jb{\nabla}^{-\l s_\al}\left(w_2^j\right)\|_{L^1_t L^{\tilde r}_x}.
\end{align*}
Then, by Lemma \ref{LEM:prod} (i) with $1<\tilde r_1<\infty$, $\frac{1}{\tilde r}= \frac{1}{\tilde r_1}+\frac{j-1}{q}$ and by Holder's inequality with $a=\frac{p}{1-\theta_\l}$, if $\theta_\l\neq 1$ and $a=\infty$ if $\theta_\l=1$, we have 
\begin{align*}
\|\eqref{eq:2ord-z}\|_{X^s}
& \lesssim \|\Xi^2\|_{\mathcal Z_2}\sum_{\l=1}^{2m}\sum_{j=1}^{2m+1-\l} T^{1-\frac{j-\theta_\l}{p}}\|\jb{\nabla}^{-\l s_\al}w_2\|_{L^a_T L^{\tilde r_1}_x}\|w_2\|^{j-1}_{L^p_TL^{q}_x}.
\end{align*}
Our definition of $\theta_\l$ guarantees that
\begin{equation*}
 -\l s_\al=\theta_\l s+(1-\theta) (s-\gamma_{p,q}), 
\end{equation*}
and so with $1<\tilde r_1<\infty$ such that
$\frac{1}{\tilde r_1}=\frac{\theta_\l}{2}+\frac{1-\theta_\l}{q},$
we get using the interpolation of Sobolev spaces that:
\begin{align}
\|\eqref{eq:2ord-z}\|_{X^s}
& \lesssim \|\Xi^2\|_{\mathcal Z_2}\sum_{\l=1}^{2m}\sum_{j=1}^{2m+1-\l} T^{1-\frac{j-\theta_\l}{p}}\
\|w_2\|_{L^\infty_TH^s_x}^{\theta_\l} \|w_2\|^{(1-\theta_\l)+j-1}_{L^p_TW^{s-\gamma_{p,q},q}_x}\notag\\
&\lesssim \|\Xi^2\|_{\mathcal Z_2}T^{a} (\|w_2\|_{X^s(T)}+\dots + \|w_2\|_{X^s(T)}^{2m}),\label{eq:b}
\end{align}
where $a:=1-\frac{2m}{p}+\frac{\eps}{p\gamma_{p,q}}>0$ since $p\geq 2m$.
Note that, in order to interpolate the Sobolev space $W^{-\l s_\al, r_1}(\T^2)$ between $W^{s-\gamma_{p,q}, q}_x(\T^2)$ and $H^s(\T^2)$, we need 
\begin{equation*}
s-\gamma_{p,q}\leq -\l s_\al \leq s, \quad \textrm{ for all } \l=1,2,\dots, m.
\end{equation*}
This leads to $s>2m(1-\al)$
and to $s\leq 1-\al+\gamma_{p,q}+\eps$ which, as we have seen, are both satisfied. 

In conclusion, \eqref{eq:b} holds provided the following are satisfied: $p\geq 2m$, $q\geq 4m+2$,
and $\al>1-\frac{1}{2m^2+m+1}$, $s\in\left(s_{\rm crit}, \min\{s_1,1-\al+\gamma_{p,q}\}\right)$ for $m\geq 2$ and $\al>\frac 45$, $s\in\left(2(1-\al),\min\{s_1,1-\al+\gamma_{p,q}\}\right)$ for $m=1$.

\medskip
$\bullet$ {\bf Case 3. Estimate on \eqref{eq:2ord-a}.} For $\al>\frac 45$ if $m=1$ and $\al>1-\frac{1}{2m^2+m+1}$ if $m\geq 2$, $p>2m+1$ and $s$ satisfying
\begin{equation}\label{eq: cond s in last case}
1-\frac{\al\beta}{m}\leq s\leq\al,
\end{equation}
with $\beta:=\frac 12+\frac{2m+1}{2p} \in (\frac 12,1)$,
we choose $r_2\in [\frac{2}{1+\al-s}, 2]$ satisfying
\begin{equation*}
\frac{1+\al-s}{2}\geq\frac{1}{r_2}\geq \frac{2m+1}{q}-(2m+1)\frac{s-\gamma_{p,q}}{2}.
\end{equation*}

Using the Strichartz estimates in Lemma \ref{lemma-str} and the dual Sobolev's embedding $L^{r_2}(\T^2)\subset H^{s-\al}(\T^2)$, we have 
\begin{equation*}
\|\eqref{eq:2ord-a}\|_{X^s(T)}\lesssim  \sum_{j=1}^{2m+1}\|(\mathcal I(:\!Z^{2m+1}\!:))^{2m+1-j}w_2^j\| _{L^1_TH^{s-\al}_x}\lesssim \sum_{j=1}^{2m+1} \|(\mathcal I(:\!Z^{2m+1}\!:))^{2m+1-j}w_2^j\|_{L^1_TL^{r_2}_x}.
\end{equation*}
Since $(\mathcal I(:\!Z^{2m+1}\!:))^{2m+1-j}\in L^\infty_t W^{-(2m+1)(1-\al)+\al-\eps,\infty}_x$ for all $j=1,2,\dots,2m+1$, and by observing that $-(2m+1)(1-\al)+\al>0$, we then have
\begin{align}
\|\eqref{eq:2ord-a}\|_{X^s(T)}&\lesssim \sum_{j=1}^{2m+1} \|(\mathcal I(:\!Z^{2m+1}\!:))^{2m+1-j}\|_{L^\infty_T L^\infty_x}\|w_2\|^j_{L^{2m+1}_TL^{(2m+1)r_2}_x}\notag\\
&\lesssim T^b \sum_{j=1}^{2m+1} \|(\mathcal I(:\!Z^{2m+1}\!:))^{2m+1-j}\|_{L^\infty_T L^\infty_x}\|w_2\|^j_{L^{p}_TW^{s-\gamma_{p,q},q}_x}\notag\\
&\lesssim \|\Xi^2\|_{\mathcal Z_2}T^b \left(\|w_2\|_{L^{p}_TW^{s-\gamma_{p,q},q}_x}+\|w_2\|_{L^{p}_TW^{s-\gamma_{p,q},q}_x}^{2m+1}\right),\label{eq:c}
\end{align}
where $b>0$ and where in the second step we used the Sobolev embedding $ W^{s-\gamma_{p,q},q}(\T^2)\subset L^{(2m+1)r_2}(\T^2)$ and $p>2m+1$.
\medskip

Therefore, \eqref{eq:c} holds provided that $\al>1-\frac{1}{2m^2+m+1}$, $p>2m+1$ and provided that $s$ satisfies \eqref{eq: cond s in last case}.
Note that $\alpha>s_1$, so the condition $s\leq \alpha$ is satisfied as soon as we impose $s<s_1$. 
Note also that the condition $s<s_1$ is required when estimating \eqref{cont-map:2ord0} and \eqref{eq:2ord-z}. Combining $s<s_1$ and \eqref{eq: cond s in last case}, implies that $\alpha>1-\frac{\beta}{2m^2+m+\beta}$, which is exactly the hypothesis on $\alpha$ in Proposition \ref{LWP-Str-sec-det} in the case $m\geq 2$. 

In conclusion, by \eqref{eq:a}, \eqref{eq:b} and \eqref{eq:c}, we have that
\begin{align*}
\|w_2\|_{X^s(T)}\leq &\bigg(C_1T+C_2T^a\left(\|w_2\|_{X^s(T)}+\dots+\|w_2\|_{X^s(T)}^{2m}\right)\\
&+C_3T^b\left(\|w_2\|_{X^s(T)}+\dots+\|w_2\|_{X^s(T)}^{2m+1}\right)\bigg)\|\Xi^2\|_{\mathcal Z_2},
\end{align*}
for some $C_1,C_2,C_3>0$. By taking $R:=3C_1\|\Xi^2\|_{\mathcal Z_2}$ and $T\leq 1$ such that $C_2T^a\|\Xi^2\|_{\mathcal Z_2}(1+R+\dots +R^{2m-1})\leq \frac 13$ and $C_3T^b\|\Xi^2\|_{\mathcal Z_2}(1+R+\dots +R^{2m})\leq \frac 13$, we get that  the map $\Psi$ defined in \eqref{cont-map:2ord} maps the ball $B(0,R)$ in $X^s(T)$ to itself. Similarly, it can be shown that $\Psi$ is a contraction on $B(0,R)$.

\end{proof}

\appendix

\section{Proof of Proposition \ref{LWP-Str}}
\noi By recalling that the map \eqref{cont-map} can be decomposed as the sum
\[
\eqref{cont-map}=\eqref{cont-map1}+\eqref{cont-map2}+\eqref{cont-map3}
\]
we distinguish three cases. 

\medskip

\noi
$\bullet$  {\bf Case 1. Estimate on term \eqref{cont-map1}.}

As in the proof of Proposition \ref{LWP-Str-sec}, Case 1, by the Strichartz estimates in Lemma \ref{lemma-str} and the embedding of Lebesgue spaces on finite measure domains, we have
\begin{align*}
\|\eqref{cont-map1}\|_{X^{s}}&\lesssim \|\Xi^1_{2m+1}\|_{L^1_t H^{s-\al}_x}\lesssim T\|\jb{\nb}^{s-\al}\Xi^1_{2m+1}\|_{L^\infty_{t}L^\infty_x}\\
&\lesssim T\|\jb{\nb}^{-(2m+1)(1-\al)-\eps}\Xi^1_{2m+1}\|_{L^\infty_{t}L^\infty_x}\lesssim T\|\Xi^1\|_{\mathcal Z_1},
\end{align*}
provided that
\begin{equation}\label{ch5:cond3}
s-\al<-(2m+1)(1-\al) \iff s<\al(2m+2)-(2m+1):=\tilde s_1.
\end{equation}

Note that condition \eqref{ch5:cond3} and condition $s>s_{\rm crit}$ in \eqref{ch5:cond0} imply
that $\al>1-\frac{1}{2m^2+2m+1}$.
We remark here that \eqref{ch5:cond3} is a stronger assumption than \eqref{eq:source-2ord}, so \eqref{ch5:cond3} is responsible for the narrower range of $\al$ in the first order expansion well-posedness theory (in the case $m\geq 2$) compared to the second order expansion well-posedness theory. 

\medskip

\noi
$\bullet$  {\bf Case 2. Estimate on \eqref{cont-map2}.}
In this case we argue exactly as in Case 2 in the proof of Proposition \ref{LWP-Str-sec}. 
More precisely, we notice that \eqref{cont-map2} has the same form as \eqref{eq:2ord-z} in which we fix $j=2m+1-\ell$.
As in the proof of Proposition \ref{LWP-Str-sec}, in addition to the assumptions on $\al$ and $s$ in Case 1, in Case 2 we further need to impose that $s<1-\al+\gamma_{p,q}$ and that, for $m=1$, $s>2(1-\al)$ (this insures that $\theta_\ell\leq 1$). 
Note that, for $m=1$, by combining $s>2(1-\al)$ with the condition $s<\tilde s_1$ in \eqref{ch5:cond3}, 
we get the restriction $\al>\frac 56$. 

\medskip
\noi
$\bullet$  {\bf Case 3. Estimate on \eqref{cont-map3}.}
In this case we argue exactly as in Case 3 in the proof of Proposition \ref{LWP-Str-sec}. 
More precisely, we notice that \eqref{cont-map3} has the same form as \eqref{eq:2ord-a} in which we fix $j=2m+1$. 
As in the proof of Proposition \ref{LWP-Str-sec}, in Case 3 we further need to impose that 
$s>1-\frac{\al\beta}{m}$, where $\beta:=\frac 12+\frac{2m+1}{2p}$. Combining this with the condition $s<\tilde s_1$ in \eqref{ch5:cond3}, leads to 
$\alpha>1-\frac{\beta}{2m^2+2m+\beta}$, which is exactly the hypothesis on $\alpha$ in Proposition \ref{LWP-Str} in the case $m\geq 2$.
For $m=1$, the condition $\alpha>\frac 56$ is more restrictive than the above mentioned condition on $\alpha$ (since $p<5$ and so $\beta>\frac 45$).

\begin{ackno}\rm

The authors would like to thank Tadahiro Oh and Younes Zine for helpful discussions. The authors are grateful to Nikolay Tzvetkov for a remark on an earlier version of the paper which allowed them to relax some of the hypotheses in the results.
L.F. and O.P. were supported by 
the EPSRC New Investigator Award 
(grant no. EP/S033157/1).
\end{ackno}


\begin{bibdiv}
\begin{biblist}

\bib{BG}{article}{
   author={Barashkov, N.},
   author={Gubinelli, M.},
   title={A variational method for  $\Phi^4_3$},
   journal={Duke Math. J.},
   volume={169},
   date={2020},
   number={17},
   pages={3339--3415},
   issn={0012-7094},
}

\bib{BenOh}{article}{
   author={B\'{e}nyi, \'{A}.},
   author={Oh, T.},
   title={ Modulation spaces, Wiener amalgam spaces, and Brownian motions},
   journal={Adv. Math.},
   volume={228},
   date={2011},
   number={5},
   pages={2943--2981},
   issn={0001-8708},
}

\bibitem{BOZ} \'{A}. B\'{e}nyi, T. Oh, T. Zhao, {\it Fractional Leibniz rule on the torus}, Proc. Amer. Math. Soc. 153 (2025), no. 1, 207-221. 
		
\bib{Bog}{book}{
   author={Bogachev, V. I.},
   title={Gaussian measures},
   series={Mathematical Surveys and Monographs},
   volume={62},
   publisher={Amer. Math. Soc., Providence, RI},
   date={1998},
   pages={xii+433},
   isbn={0-8218-1054-5},
}

\bib{BD}{article}{
   author={Bou\'{e}, M.},
   author={Dupuis, P.},
   title={A variational representation for certain functionals of Brownian
   motion},
   journal={Ann. Probab.},
   volume={26},
   date={1998},
   number={4},
   pages={1641--1659},
   issn={0091-1798},
}


\bib{Bou-CMP94}{article}{
   author={Bourgain, J.},
   title={Periodic nonlinear Schr\"{o}dinger equation and invariant measures},
   journal={Comm. Math. Phys.},
   volume={166},
   date={1994},
   number={1},
   pages={1--26},
   issn={0010-3616},
}

\bib{Bou-CMP96}{article}{
   author={Bourgain, {J.}},
   title={Invariant measures for the $2$D-defocusing nonlinear Schr\"{o}dinger
   equation},
   journal={Comm. Math. Phys.},
   volume={176},
   date={1996},
   number={2},
   pages={421--445},
   issn={0010-3616},
}

\bib{BringmanI}{article}{
   author={Bringman, B.},
    title={Invariant Gibbs measures for the three-dimensional wave equation with a Hartree nonlinearity I: Measures}, 
  journal={Stoch. Partial Differ. Equ. Anal. Comput.},
  volume={10},
   date={2022},
   number={1},
   pages={1--89},
}

\bib{BringmanII}{article}{
   author={Bringman, B.},
    title={Invariant Gibbs measures for the three-dimensional wave equation with a Hartree nonlinearity II: Dynamics}, 
  journal={J. Eur. Math. Soc.},
  volume={26},
   date={2024},
   number={6},
   pages={1933--2098},
  
  }
  
  \bib{BDNY}{article}{
   author={Bringman, B.},
     author={Deng, Y.},
       author={Nahmod, A.},
         author={Yue, H.},
    title={Invariant Gibbs measures for the three dimensional cubic nonlinear wave equation}, 
  journal={Invent. Math.},
  volume={236},
   date={2024},
   pages={1133--1411},

}

\bib{BurqTz}{article}{
   author={Burq, N.},
   author={Tzvetkov, N.},
   title={Random data Cauchy theory for supercritical wave equations. I.
   Local theory},
   journal={Invent. Math.},
   volume={173},
   date={2008},
   number={3},
   pages={449--475},
   issn={0020-9910},
}

\bib{DPD}{article}{
   author={Da Prato, {G.}},
   author={Debussche, A.},
   title={Two-dimensional Navier-Stokes equations driven by a space-time
   white noise},
   journal={J. Funct. Anal.},
   volume={196},
   date={2002},
   number={1},
   pages={180--210},
   issn={0022-1236},
}		
						
\bib{DPT}{article}{
   author={Da Prato, G.},
   author={Tubaro, L.},
   title={Wick powers in stochastic PDEs: an introduction},
   journal={Technical Report UTM, 2006},
   date={2006},
   pages={39 pp}
}

\bib{VDD-JDE}{article}{
   author={Dinh, V. D.},
   title={Strichartz estimates for the fractional Schr\"{o}dinger and wave
   equations on compact manifolds without boundary},
   journal={J. Differential Equations},
   volume={263},
   date={2017},
   number={12},
   pages={8804--8837},
   issn={0022-0396},
}

\bib{FV}{book}{
   author={Friz, P. K.},
   author={Victoir, N. B.},
   title={Multidimensional stochastic processes as rough paths},
   series={Cambridge Studies in Advanced Mathematics},
   volume={120},
   note={Theory and applications},
   publisher={Cambridge University Press, Cambridge},
   date={2010},
   pages={xiv+656},
   isbn={978-0-521-87607-0},
}
		
\bib{GJ}{book}{
   author={Glimm, J.},
   author={Jaffe, A.},
   title={Quantum physics},
   edition={2},
   note={A functional integral point of view},
   publisher={Springer-Verlag, New York},
   date={1987},
   pages={xxii+535},
   isbn={0-387-96476-2},
}

\bib{GIP}{article}{
   author={Gubinelli, M.},
   author={Imkeller, P.},
   author={Perkowski, N.},
   title={Paracontrolled distributions and singular PDEs},
   journal={Forum Math. Pi},
   volume={3},
   date={2015},
   number={e6},
   
   }

\bib{GKO}{article}{
   author={Gubinelli, M.},
   author={Koch, H.},
   author={Oh, T.},
   title={Renormalization of the two-dimensional stochastic nonlinear wave
   equations},
   journal={Trans. Amer. Math. Soc.},
   volume={370},
   date={2018},
   number={10},
   pages={7335--7359},
   issn={0002-9947},
}

\bib{GKO2}{article}{
   author={Gubinelli, {M.}},
   author={Koch, H.},
   author={Oh, T.},
   title={Paracontrolled approach to the
three-dimensional stochastic nonlinear wave equation
with quadratic nonlinearity},
   journal={J. Eur. Math.Soc.},
   volume={26},
   date={2024},
   number={3},
   pages={817--874},
}


 \bib{GKOT}{article}{
   author={Gubinelli, M.},
   author={Koch, H.},
   author={Oh, T.},
   author={Tolomeo, L.},
   title={Global dynamics for the two-dimensional stochastic nonlinear wave
   equations},
   journal={Int. Math. Res. Not. IMRN},
   date={2022},
   number={21},
   pages={16954--16999},
   issn={1073-7928},
}


\bib{GOTW}{article}{
   author={Gunaratnam, T.},
   author={Oh, T.},
   author={Tzvetkov, N.},
   author={Weber, H.},
   title={Quasi-invariant Gaussian measures for the nonlinear wave equation
   in three dimensions},
   journal={Probab. Math. Phys.},
   volume={3},
   date={2022},
   number={2},
   pages={343--379},
   issn={2690-0998},
}

\bib{LiWa}{article}{
   author={Liang, R.},
   author={Wang, Y.},
   title={Gibbs measure for the focusing fractional NLS on the torus},
   journal={SIAM J. Math. Anal.},
   volume={54},
   date={2022},
   number={6},
   pages={6096--6118},
   issn={0036-1410},
}

\bib{McKean}{article}{
   author={McKean, H. P.},
   title={Statistical mechanics of nonlinear wave equations. IV. Cubic
   Schr\"{o}dinger},
   journal={Comm. Math. Phys.},
   volume={168},
   date={1995},
   number={3},
   pages={479--491},
   issn={0010-3616},
}		

\bib{MWX}{article}{
   author={Mourrat, J.-C.},
   author={Weber, H.},
   author={Xu, W.},
   title={Construction of $\Phi^4_3$ diagrams for pedestrians},
   conference={
      title={From particle systems to partial differential equations},
   },
   book={
      series={Springer Proc. Math. Stat.},
      volume={209},
      publisher={Springer, Cham},
   },
   date={2017},
   pages={1--46},
}

\bib{Nelson}{article}{
   author={Nelson, E.},
   title={A quartic interaction in two dimensions},
   conference={
      title={Mathematical Theory of Elementary Particles},
      address={Proc. Conf., Dedham, Mass.},
      date={1965},
   },
   book={
      publisher={M.I.T. Press, Cambridge, Mass.},
   },
   date={1966},
   pages={69--73},
}

\bib{OhOk}{article}{
   author={Oh, {T.}},
   author={Okamoto, {M.}},
   title={Comparing the stochastic nonlinear wave and heat equations: a case
   study},
   journal={Electron. J. Probab.},
   volume={26},
   date={2021},
   pages={Paper No. 9, 44},
}


 \bib{OOT}{article}{
   author={Oh, {T.}},
   author={Okamoto, {M.}},
   author={Tolomeo, L.},
   title={Focusing $\Phi_3^4$-model with a Hartree-type nonlinearity},
   journal={Mem. Amer. Math. Soc.},
   volume={204},
   date={2024},
    number={1529},
   }
 
 
 
 \bib{OOT2}{article}{
   author={Oh, {T.}},
   author={Okamoto, {M.}},
   author={Tolomeo, L.},
   title={Stochastic quantization of the $\Phi_3^3$-model},
   journal={Mem. Eur. Math. Soc.},
   volume={16},
   date={2025},
   }
 

\bib{OOTz}{article}{
   author={Oh, T.},
   author={Okamoto, M.},
   author={Tzvetkov, N.},
   title={Uniqueness and non-uniqueness of the Gaussian free field evolution under the two-dimensional Wick ordered cubic wave equation},
   journal={Ann. Inst. Henri Poincar\'e Probab. Stat.},
    volume={60},
   date={2024},
   number={3},
   pages={1684-1728},
}


\bib{OQ}{article}{
   author={Oh, T.},
   author={Quastel, J.},
   title={On the Cameron-Martin theorem and almost-sure global existence},
   journal={Proc. Edinb. Math. Soc. (2)},
   volume={59},
   date={2016},
   number={2},
   pages={483--501},
   issn={0013-0915},
}

\bib{OPT}{article}{
   author={Oh, T.},
   author={Pocovnicu, O.},
   author={Tzvetkov, N.},
   title={  Probabilistic local Cauchy theory of the cubic nonlinear wave
   equation in negative Sobolev spaces},
   journal={Ann. Inst. Fourier (Grenoble)},
   volume={72},
   date={2022},
   number={2},
   pages={771--830},
   issn={0373-0956},
}
 
\bib{OST2020}{article}{
author={Oh, T.},
author={Seong, K.},
author={Tolomeo, L.},
title = {A remark on Gibbs measures with log-correlated
Gaussian fields},
journal={Forum Math. Sigma},
 volume={12},
date={2024},
number={e50},
}		

\bibitem{ORT}  T. Oh, G. Richards, L. Thomann, 
{\it On invariant Gibbs measures for the generalized KdV equations}, Dyn. Partial Differ. Equ. 13 (2016), no. 2, 133-153. 

\bib{ROT}{article}{
   author={Oh, T.},
   author={Robert, T.},
   author={Tzvetkov, N.},
   title={Stochastic nonlinear wave dynamics on compact surfaces},
   journal={Ann. H. Lebesgue},
   volume={6},
   date={2023},
   number={6},
   pages={161--223},
}

\bib{OT2018}{article}{
   author={Oh, {T.}},
   author={Thomann, L.},
   title={A pedestrian approach to the invariant Gibbs measures for the
   2-$d$ defocusing nonlinear Schr\"{o}dinger equations},
   journal={Stoch. Partial Differ. Equ. Anal. Comput.},
   volume={6},
   date={2018},
   number={3},
   pages={397--445},
   issn={2194-0401},
}

\bib{OT2020}{article}{
   author={Oh, T.},
   author={Thomann, L.},
   title={Invariant Gibbs measures for the $2$-$d$ defocusing nonlinear wave
   equations},
   journal={Ann. Fac. Sci. Toulouse Math. (6)},
   volume={29},
   date={2020},
   number={1},
   pages={1--26},
   issn={0240-2963},
}

\bib{OWZ}{article}{
   author={Oh, T.},
   author={Wang, Y.},
   author={Zine, Y.},
   title={Three-dimensional stochastic cubic nonlinear wave equation with
   almost space-time white noise},
   journal={Stoch. Partial Differ. Equ. Anal. Comput.},
   volume={10},
   date={2022},
   number={3},
   pages={898--963},
   issn={2194-0401},
}


\bib{Oh-Zine2022}{article}{
   author={Oh, T.},
   author={Zine, Y.},
   title={A note on products of stochastic objects},
   journal={Kyoto J. Math, to appear},
}


		

\bib{Tz2008}{article}{
   author={Tzvetkov, N.},
   title={Invariant measures for the defocusing nonlinear Schr\"{o}dinger
   equation},
   journal={Ann. Inst. Fourier (Grenoble)},
   volume={58},
   date={2008},
   number={7},
   pages={2543--2604},
   issn={0373-0956},
}
		
\bib{Tz2010}{article}{
   author={Tzvetkov, {N.}},
   title={Construction of a Gibbs measure associated to the periodic
   Benjamin-Ono equation},
   journal={Probab. Theory Related Fields},
   volume={146},
   date={2010},
   number={3-4},
   pages={481--514},
   issn={0178-8051},
}

\bib{Shige}{book}{
   author={Shigekawa, I.},
   title={Stochastic analysis},
   series={Translations of Mathematical Monographs},
   volume={224},
   note={Translated from the 1998 Japanese original by the author;
   Iwanami Series in Modern Mathematics},
   publisher={American Mathematical Society, Providence, RI},
   date={2004},
   pages={xii+182},
   isbn={0-8218-2626-3},
}

\bib{Simon}{book}{
   author={Simon, B.},
   title={The $P(\phi )_{2}$ Euclidean (quantum) field theory},
   note={Princeton Series in Physics},
   publisher={Princeton University Press, Princeton, N.J.},
   date={1974},
   pages={xx+392}
}

\bib{STX}{article}{
   author={Sun, C.},
   author={Tzvetkov, N.},
   author={Xu, W.},
   title={Weak universality results for a class of
nonlinear wave equations},
   journal={arXiv:2206.05945 [math.AP], to appear in Ann. Inst. Fourier},
   date={2022}
}


\bib{Ustunel}{article}{
   author={\"{U}st\"{u}nel, A. S.},
   title={Variational calculation of Laplace transforms via entropy on
   Wiener space and applications},
   journal={J. Funct. Anal.},
   volume={267},
   date={2014},
   number={8},
   pages={3058--3083},
   issn={0022-1236}
}


\end{biblist}
\end{bibdiv}
\end{document}